\documentclass[11pt]{article}
\usepackage{amssymb,amsthm,authblk}
\usepackage{amsmath,latexsym,epsf}
\usepackage[sort&compress,numbers]{natbib}
\usepackage{changepage}
\usepackage{mathrsfs,enumitem}
\usepackage[all]{xy}
\numberwithin{equation}{section}
\usepackage{color}
\usepackage[colorlinks,linkcolor=blue,anchorcolor=orange,
citecolor=red,urlcolor=green]{hyperref}
\usepackage[mathlines,displaymath]{lineno}
\let\oldalign\align
\let\oldendalign\endalign

\renewenvironment{align}
{\linenomathNonumbers\oldalign}
{\oldendalign\endlinenomath}

\begin{document}
\theoremstyle{definition}
\newcommand{\nc}{\newcommand}
\def\PP#1#2#3{{\mathrm{Pres}}^{#1}_{#2}{#3}\setcounter{equation}{0}}
\def\mr#1{{{\mathrm{#1}}}\setcounter{equation}{0}}
\def\mc#1{{{\mathcal{#1}}}\setcounter{equation}{0}}
\def\mb#1{{{\mathbb{#1}}}\setcounter{equation}{0}}
\def\Mcc{\mc{C}}
\def\Mbe{\mb{E}}
\def\Mcp{\mc{P}}
\def\Mcg{\mc{G}}
\def\Mcx{\mc{X}}
\def\QP{\mc{QP}(\xi)}
\def\QI{\mc{QI}(\xi)}
\def\id#1{\xi\text{-id}{\left(#1\right)}}
\def\pd#1{\xi\text{-pd}{\left(#1\right)}}
\def\Qgpd#1{\xi\text{-}\mc{QG}\text{pd}{\left(#1\right)}}
\def\Qgid#1{\xi\text{-}\mc{QG}\text{id}{\left(#1\right)}}
\def\Qid#1{\xi\text{-}\mc{Q}\text{id}{\left(#1\right)}}
\def\Qpd#1{\xi\text{-}\mc{Q}\text{pd}{\left(#1\right)}}
\def\qgpd#1{\mc{QG}\text{pd}{\left(#1\right)}}
\def\qgid#1{\mc{QG}\text{id}{\left(#1\right)}}
\def\QGpdR#1{\mc{QG}\text{pd}_R{\left(#1\right)}}
\def\pdR#1{\text{pd}_R{\left(#1\right)}}
\def\QpdR#1{\mc{Q}\text{pd}_R{\left(#1\right)}}

\def\QGP{\mc{QGP}(\xi)}
\def\QGI{\mc{QGI}(\xi)}
\def\qGP{\mc{QGP}}
\def\qGI{\mc{QGI}}
\def\extri{(\mc{C},\mb{E},\mathfrak{s})}
\def\GP{\mc{G}\mc{P}(\xi)}
\def\GI{\mc{G}\mc{I}(\xi)}
\def\P{\mc{P}(\xi)}
\def\Extri{\mb{E}\text{-triangle}}
\def\ext{\xi \text{xt}_{\xi}}
\def\Ext{\text{Ext}_{R}}
\def\Extqgp{\text{Ext}_{\qGP}}
\def\Extqgi{\text{Ext}_{\qGI}}
\def\QGext{\mc{QG}\text{Ext}_R}
\def\>{\longrightarrow}

\newtheorem{defn}{\bf Definition}[section]
\newtheorem{cor}[defn]{\bf Corollary}
\newtheorem{prop}[defn]{\bf Proposition}
\newtheorem{thm}[defn]{\bf Theorem}
\newtheorem{lem}[defn]{\bf Lemma}
\newtheorem{rem}[defn]{\bf Remark}
\newtheorem{exam}[defn]{\bf Example}
\newtheorem{fact}[defn]{\bf Fact}
\newtheorem{cond}[defn]{\bf Condition}
\def\bskip#1{{\vskip 20pt}\setcounter{equation}{0}}
\def\sskip#1{{\vskip 5pt}\setcounter{equation}{0}}
\newcommand{\myvskip}{\par\vspace{10pt}}
\def\bg#1{\begin{#1}\setcounter{equation}{0}}
	\def\ed#1{\end{#1}\setcounter{equation}{0}}

\title{Relative quasi-Gorenstein homological dimensions in
extriangulated categories\footnote{Supported by the Yunnan Provincial Department of Education Science Research Fund Project (Project No. 2026J2379)}}
\renewcommand{\Authfont}{\small\normalfont}
\renewcommand{\Affilfont}{\small\normalfont}

\author[a,b]{Zhenggang He\thanks{Corresponding author: zhenggang\_he@163.com}}
\author[c]{Jifen Liu}
\author[a]{Jiaqun Wei}

\affil[a]{School of Mathematical Sciences,
Zhejiang Normal University, Jinhua 321004, China}

\affil[b]{Department of Public Foundation,
Dali Nursing Vocational College, Dali 671006, China}

\affil[c]{School of Mathematical Sciences,
Laboratory of Mathematics and Complex Systems, MOE,
Beijing Normal University, Beijing 100875, China}
\date{}
\maketitle
\baselineskip 15pt
%
%
\vspace{-1cm}
\vskip 10pt%
\noindent{\bf Abstract}:
Let $(\mathcal C,\mathbb E,\mathfrak s)$ be an
extriangulated category equipped with a proper class
$\xi$ of $\mathbb E$-triangles.
We establish stability under iteration for the quasi-$\xi$-Gorenstein
projective and injective subcategories and extend the relative
Ext-vanishing criteria for their finite homological dimensions to larger
coefficient subcategories. An example in a bounded homotopy category
shows that
$\mathcal{P}(\xi)\subsetneq\mathcal{QGP}(\xi)\subsetneq\mathcal{GP}(\xi)$.
For module categories, every quasi-Gorenstein projective
module over a left perfect ring or a commutative Noetherian
ring of finite Krull dimension is projective.
When a commutative ring admits a finite chain of trace and
nilpotent quotients ending in an Artinian ring and the
endomorphism rings at the trace steps have finite right
global dimension, we construct finite families of
quasi-projective test modules that detect prescribed
upper bounds on projective dimension for modules of
bounded cardinality through Ext vanishing in finitely
many consecutive degrees.
\myvskip


\myvskip

%
%
\vskip 30pt

\section{Introduction}
\quad~Extriangulated categories were introduced by Nakaoka and Palu as a unified framework that generalizes both exact and triangulated categories \cite{HY}. This formalism consolidates numerous classical results from these two distinct settings into a single, coherent theory. Moreover, it reveals a broader landscape of homological algebra by encompassing concrete examples that are neither exact nor triangulated \cite{HY, Zhou}. Thus, the study of extriangulated categories not only bridges previously separate theories but also provides a versatile foundation for further structural exploration.

Hu, Zhang and Zhou \cite{JDP} developed a theory of relative homological algebra in extriangulated categories. This theory unifies two existing frameworks: Beligiannis's relative homological algebra in triangulated categories \cite{BEL} and the classical homological algebra in exact categories in the sense of Quillen. The unification is achieved by fixing a suitable class $\xi$ of $\mathbb{E}$-triangles, known as a proper class. This class $\xi$ forms the foundation for defining and studying relative exact sequences, $\xi$-projective objects, $\xi$-injective objects, and the related homological dimensions.

In the category of modules, Auslander and Bridger \cite{Aus} first introduced the notion of modules of G-dimension zero, generalizing the class of finitely generated projective modules over a commutative Noetherian ring. To extend this idea beyond the finitely generated setting, Enochs and Jenda \cite{Enochs} subsequently defined Gorenstein projective modules over arbitrary associative rings, which are not necessarily finitely generated, and dually introduced Gorenstein injective modules. This laid the foundation for what is now known as Gorenstein homological algebra, a field that has since been extensively developed, see for example \cite{Avramov,Enochs1,LW,Holm}.

Building on the notion of modules of G-dimension zero, Beligiannis  introduced $\mathcal{X}$-Gorenstein objects in an additive category $\mathcal{C}$ in \cite{Beligiannis}, where $\mathcal{X}$ is a contravariantly finite subcategory of $\mathcal{C}$ and every $\mathcal{X}$-epic has a kernel in $\mathcal{C}$. Asadollahi and Salarian \cite{Asadollahi} and Ren and
Liu \cite{RL} extended Gorenstein homological algebra
to triangulated categories by adapting the approach
of Enochs and Jenda \cite{Enochs1} from abelian
categories. They introduced
$\mathcal{E}$-Gorenstein projective and
$\mathcal{E}$-Gorenstein injective objects and the
corresponding homological dimensions.
In the extriangulated setting, Hu, Zhang and
Zhou \cite{JDP} introduced $\xi$-Gorenstein projective
and $\xi$-Gorenstein injective dimensions and studied
their properties. These dimensions were further
characterized in terms of derived functors
in \cite{JZP}.

In the module category over a ring $R$, the concept of a quasi-projective module was introduced by Wu and Jans \cite{LET}. Specifically, an $R$-module $M$ is said to be quasi-projective if for any $R$-module $A$, any epimorphism $g : M \to A$, and any morphism $f : M \to A$, there exists an endomorphism $\phi \in \operatorname{End}_R(M)$ such that $f = g\phi$. This notion was later extended to general Abelian categories by Rangaswamy and Vanaja \cite{KM}, who introduced the quasi-projective cover of an object. They showed that, in an Abelian category with enough projectives, the existence of quasi-projective covers for all objects guarantees the existence of projective covers. Further results and developments on quasi-projective modules and related structures can be found in \cite{Baba, Fuchs, KM, LET}.

Mohammadi Aghjeh Mashhad introduced quasi-Gorenstein
projective and injective modules in \cite{AF} and
studied their behavior in short exact sequences and
their finite homological dimensions. The relative
theory in extriangulated categories was developed
in \cite{He}, including closure properties,
mixed-resolution characterizations and Ext-vanishing
criteria for finite quasi-Gorenstein dimensions.

In this paper, we prove stability under iteration
for the quasi-$\xi$-Gorenstein projective and injective
subcategories (see Theorem~\ref{nice}).
The coefficient subcategories $\mathcal{QP}(\xi)$ and
$\mathcal{QI}(\xi)$ remain fixed throughout the
respective iterations; see \cite{JDPPR} for related
stability results.
For objects of finite quasi-$\xi$-Gorenstein
projective dimension, we show that the coefficient
subcategory in the Ext-vanishing criterion
of \cite{He} can be restricted to $\mathcal{P}(\xi)$
or enlarged to $\widehat{\mathcal{QP}}(\xi)$
(see Theorem~\ref{thm4.4}).
The corresponding injective criteria are obtained
dually (see Theorem~\ref{thm4.8}).

In the bounded homotopy category of finitely generated
modules over the dual numbers, we consider the relative
structure determined by short exact sequences in
cohomology. The inclusions
$\mathcal{P}(\xi)\subseteq\mathcal{QGP}(\xi)
\subseteq\mathcal{GP}(\xi)$ are both strict, and the
subcategories $\mathcal{QP}(\xi)$ and
$\mathcal{QGP}(\xi)$ are incomparable. The induced
extriangulated structure is neither exact nor
triangulated (see Example~\ref{separatingexample}).

For module categories with their usual exact structure,
we prove that $\mathcal{QGP}(R)=\operatorname{Proj}R$
whenever $R$ is left perfect or commutative Noetherian
of finite Krull dimension
(see Theorems~\ref{thm:perfect} and~\ref{thm:finite-krull}).

Test modules for projectivity and finite homological
dimensions are studied in \cite{Jothilingam,CDT}
for finitely generated modules over commutative
Noetherian rings. Let $R$ be a commutative ring,
$T$ a finitely generated $R$-module and $N$ an
$R$-module. If $\operatorname{Hom}_R(T,N)$ has finite
projective dimension as a right
$\operatorname{End}_R(T)$-module, then there exists
a bounded augmented complex ending in $N$, with
all other terms in $\operatorname{Add}_R T$ and
homology modules annihilated by $\operatorname{tr}_R(T)$
(see Lemma~\ref{lem:trace}).
This construction yields a reduction modulo the
trace ideal for the coefficient class associated
to a complete quasi-projective resolution
(see Proposition~\ref{prop:reduction}).

Suppose now that a commutative ring $R$ admits a
finite chain of quotients by trace ideals and
nilpotent ideals ending in an Artinian ring.
Assume that the modules defining the trace ideals
are nonzero and finitely generated and that their
endomorphism rings have finite right global
dimension. For each infinite cardinal
$\lambda\geq|R|$, we construct a finite family
$\mathscr Q_\lambda$ of quasi-projective $R$-modules
and determine an integer $c\geq0$ such that, for
any $R$-module $X$ with $|X|\leq\lambda$ and any
integer $n\geq0$, $\operatorname{pd}_R X\leq n$
if and only if $\operatorname{Ext}_R^{n+i}(X,Q)=0$
for all $Q\in\mathscr Q_\lambda$ and $1\leq i\leq c+1$
(see Theorem~\ref{thm:finite-test}).
Here $\mathscr Q_\lambda$ and $c$ are independent
of $X$ and $n$, and $X$ need not be finitely
generated or have finite quasi-Gorenstein
projective dimension.
In particular, $\mathcal{QGP}(R)=\operatorname{Proj}R$.
For $R=k[[x,y]]/(xy)$, where $k$ is a field,
the test family is described explicitly in
Example~\ref{ex:node-test}, and three consecutive
Ext degrees suffice.

The paper is organized as follows.
In Section~2, we recall some basic definitions and results.
In Section~3, we establish the stability of the
quasi-$\xi$-Gorenstein projective and injective
subcategories under iteration.
In Section~4, we study the corresponding relative
homological dimensions.
In Section~5, we give applications to module categories,
including projectivity criteria and finite families
of quasi-projective test modules.

\section{Preliminaries}
\quad~Throughout the paper,
$\Mcc=(\Mcc,\Mbe,\mathfrak{s})$ denotes an
extriangulated category equipped with a proper class
$\xi$ of $\Mbe$-triangles. We assume that $\Mcc$ has
enough $\xi$-projective and $\xi$-injective objects and
satisfies Condition~\ref{WIC}. All subcategories are
assumed to be full and closed under isomorphisms.

We use the terminology of Nakaoka and Palu
\cite{HY}. Thus $\Mbe:\Mcc^{\mathrm{op}}\times\Mcc
\to\mathbf{Ab}$ is a biadditive functor and
$\mathfrak{s}$ is an additive realization satisfying
the axioms in \cite[Definition~2.12]{HY}. If
$\delta\in\Mbe(C,A)$ is realized by a sequence
$A\xrightarrow{x}B\xrightarrow{y}C$, we write
\[
    A\xrightarrow{x}B\xrightarrow{y}C
    \overset{\delta}{\dashrightarrow}
\]
and call it an $\Mbe$-triangle. The morphisms $x$ and
$y$ are called an inflation and a deflation,
respectively. The extension $\delta$ will be omitted
when it is not needed. A morphism of $\Mbe$-triangles is
a commutative diagram
\[
    \xymatrix{
        A\ar[r]^x\ar[d]^a
        &B\ar[r]^y\ar[d]^b
        &C\ar[d]^c\ar@{-->}[r]^{\delta}
        &{}\\
        A'\ar[r]^{x'}
        &B'\ar[r]^{y'}
        &C'\ar@{-->}[r]^{\delta'}
        &{}
    }
\]
for which $(a,c)$ is a morphism of the corresponding
$\Mbe$-extensions.

We shall use the following standard facts.

\begin{lem}[{\cite[Corollary~3.5]{HY}}]\label{FJ}
Let
\[
    \xymatrix{
        A\ar[r]^x\ar[d]^a
        &B\ar[r]^y\ar[d]^b
        &C\ar[d]^c\ar@{-->}[r]^{\delta}
        &{}\\
        A'\ar[r]^{x'}
        &B'\ar[r]^{y'}
        &C'\ar@{-->}[r]^{\delta'}
        &{}
    }
\]
be a morphism of $\Mbe$-triangles. The following
conditions are equivalent:

(1) $a$  factors through $x$;

(2) $a_*\delta=c^*\delta'=0$;
    
(3) $c$ factors through $y'$.

In particular, an $\Mbe$-triangle
$A\xrightarrow{x}B\xrightarrow{y}C\dashrightarrow$
splits if and only if $x$ is a section, or equivalently,
if and only if $y$ is a retraction.
\end{lem}

\begin{lem}[{\cite[Proposition~3.15]{HY}}]\label{BH}
The following statements hold.

\textup{(1)}
Let $ A_1\xrightarrow{x_1}B_1\xrightarrow{y_1}C
    \dashrightarrow$ and $A_2\xrightarrow{x_2}B_2\xrightarrow{y_2}C
    \dashrightarrow$
be $\Mbe$-triangles. There exists an object $M$ and a
commutative diagram
\[
    \xymatrix{
        &A_2\ar[d]^{m_2}\ar@{=}[r]
        &A_2\ar[d]^{x_2}\\
        A_1\ar@{=}[d]\ar[r]^{m_1}
        &M\ar[r]^{e_1}\ar[d]^{e_2}
        &B_2\ar[d]^{y_2}\\
        A_1\ar[r]^{x_1}
        &B_1\ar[r]^{y_1}
        &C
    }
\]
such that
$A_1\xrightarrow{m_1}M\xrightarrow{e_1}B_2
\dashrightarrow$
and
$A_2\xrightarrow{m_2}M\xrightarrow{e_2}B_1
\dashrightarrow$
are $\Mbe$-triangles.

\textup{(2)}
Let $A\xrightarrow{x_1}B_1\xrightarrow{y_1}C_1
    \dashrightarrow$and $A\xrightarrow{x_2}B_2\xrightarrow{y_2}C_2
    \dashrightarrow$
be $\Mbe$-triangles. There exists an object $M$ and a
commutative diagram
\[
    \xymatrix{
        A\ar[d]^{x_2}\ar[r]^{x_1}
        &B_1\ar[d]^{m_2}\ar[r]^{y_1}
        &C_1\ar@{=}[d]\\
        B_2\ar[d]^{y_2}\ar[r]^{m_1}
        &M\ar[r]^{e_1}\ar[d]^{e_2}
        &C_1\\
        C_2\ar@{=}[r]
        &C_2
    }
\]
such that
$B_2\xrightarrow{m_1}M\xrightarrow{e_1}C_1
\dashrightarrow$
and
$B_1\xrightarrow{m_2}M\xrightarrow{e_2}C_2
\dashrightarrow$
are $\Mbe$-triangles.
\end{lem}
\begin{cond}[{\citep[Condition 5.8]{HY}}] (WIC Condition)\label{WIC} Let $f\in\Mcc(A,B),g\in\Mcc(B,C)$ be any pair of morphisms.
	Consider the following conditions.

	(1) If $gf$ is an inflation, then so is $f$.

	(2) If $gf$ is a deflation, then so is $g$.
\end{cond}

A class $\xi$ of $\Mbe$-triangles is proper if it is
closed under isomorphisms and finite coproducts, contains
the split $\Mbe$-triangles, is closed under base and
cobase change, and is saturated; see
\cite[Definition~3.1]{JDP}. For objects $A,C\in\Mcc$,
set
\[
    \Mbe_\xi(C,A)
    =
    \{
        \delta\in\Mbe(C,A)
        \mid
        \delta\text{ is realized by an $\Mbe$-triangle
        in }\xi
    \},
\]
and let $\mathfrak{s}_\xi$ be the restriction of
$\mathfrak{s}$ to $\Mbe_\xi$.

\begin{lem}[{\cite[Theorem~3.2]{JDP}}]\label{Thm3.2}
The triple
$(\Mcc,\Mbe_\xi,\mathfrak{s}_\xi)$ is an
extriangulated category.
\end{lem}

Thus the standard constructions for $\Mbe$-triangles
may be carried out within $\xi$. The first and second morphisms in an $\mathbb{E}$-triangle in $\xi$ are called a $\xi$-inflation and a $\xi$-deflation, respectively.

An object $P\in\Mcc$ is $\xi$-projective if every
$\Mbe$-triangle $A\longrightarrow B\longrightarrow C\dashrightarrow$
in $\xi$ induces an exact sequence
\[
    0\longrightarrow\Mcc(P,A)
    \longrightarrow\Mcc(P,B)
    \longrightarrow\Mcc(P,C)
    \longrightarrow0.
\]
The notion of a $\xi$-injective object is defined
dually. We denote the corresponding subcategories by
$\P$ and $\mathcal I(\xi)$. The category $\Mcc$ has
enough $\xi$-projectives if every object $A$ occurs in an
$\Mbe$-triangle $K\longrightarrow P\longrightarrow A\dashrightarrow$
in $\xi$ with $P\in\P$; enough $\xi$-injectives is
defined dually.

A complex
\[
    \mathbf X:\quad
    \cdots\longrightarrow X_1\xrightarrow{d_1}X_0
    \xrightarrow{d_0}X_{-1}\longrightarrow\cdots
\]
is $\xi$-exact if, for every $n\in\mathbb Z$, there is
an $\Mbe$-triangle
\[
    K_{n+1}\xrightarrow{g_n}X_n\xrightarrow{f_n}K_n
    \dashrightarrow
\]
in $\xi$ such that $d_n=g_{n-1}f_n$. These are called
the $\xi$-resolution $\Mbe$-triangles of $\mathbf X$.

Let $\mathcal W$ be a class of objects of $\Mcc$. An
$\Mbe$-triangle
$A\to B\to C\dashrightarrow$ in $\xi$ is
$\Mcc(-,\mathcal W)$-exact if
\[
    0\longrightarrow\Mcc(C,W)
    \longrightarrow\Mcc(B,W)
    \longrightarrow\Mcc(A,W)
    \longrightarrow0
\]
is exact for every $W\in\mathcal W$. The notion of
$\Mcc(\mathcal W,-)$-exactness is defined dually. A
$\xi$-exact complex has either property if all its
$\xi$-resolution $\Mbe$-triangles have that property.

A $\xi$-projective resolution of an object $A$ is a
right-bounded $\xi$-exact complex
\[
    \cdots\longrightarrow P_2\longrightarrow P_1
    \longrightarrow P_0\longrightarrow A
    \longrightarrow0
\]
with $P_i\in\P$. The $\xi$-projective dimension
$\pd{A}$ is the least length of a finite
$\xi$-projective resolution of $A$; we set
$\pd{0}=-1$ and $\pd{A}=\infty$ if no finite resolution
exists. The notions of a $\xi$-injective coresolution
and of $\id{A}$ are defined dually. We denote by
$\widehat{\mathcal P}(\xi)$ and
$\widehat{\mathcal I}(\xi)$ the full subcategories of
objects of finite $\xi$-projective and finite
$\xi$-injective dimension, respectively.

For objects $A,B\in\Mcc$, choose a $\xi$-projective
resolution $\mathbf P\to A$ and a $\xi$-injective
coresolution $B\to\mathbf I$. For $n\geq0$, set
\[
    \xi\mathrm{xt}_{\mathcal P(\xi)}^n(A,B)
    =
    H^n\Mcc(\mathbf P,B)
\]
and
\[
    \xi\mathrm{xt}_{\mathcal I(\xi)}^n(A,B)
    =
    H^n\Mcc(A,\mathbf I).
\]
These groups are naturally isomorphic
\cite[Definition~3.2]{JZP}; their common value is denoted
by $\ext^n(A,B)$.

\begin{lem}[{\cite[Lemma~3.4]{JZP}}]\label{LZHL}
Let $A\longrightarrow B\longrightarrow C\dashrightarrow$
be an $\Mbe$-triangle in $\xi$. For every $X\in\Mcc$,
there are long exact sequences
\[
\begin{aligned}
    0&\longrightarrow\ext^0(X,A)
    \longrightarrow\ext^0(X,B)
    \longrightarrow\ext^0(X,C)
    \longrightarrow\ext^1(X,A)
    \longrightarrow\cdots,\\
    0&\longrightarrow\ext^0(C,X)
    \longrightarrow\ext^0(B,X)
    \longrightarrow\ext^0(A,X)
    \longrightarrow\ext^1(C,X)
    \longrightarrow\cdots.
\end{aligned}
\]
There is also a natural morphism
\[
    \Mcc(A,B)\longrightarrow\ext^0(A,B),
\]
which is an isomorphism if $A\in\P$ or
$B\in\mathcal I(\xi)$.
\end{lem}

	\section{Quasi-$\xi$-Gorenstein objects}
	\qquad In this section, we review the definition of quasi-$\xi$-Gorenstein projective (resp. quasi-$\xi$-Gorenstein injective) objects and explore their fundamental properties.

	\begin{defn}[{\citep[Definition 3.1]{He}}]
		An object $Q$ in $\Mcc$ is called quasi-$\xi$-projective if for any $\mathbb{E}$-triangle
		\[
			A\longrightarrow Q \longrightarrow B  \dashrightarrow
		\]
		in $\xi$, the induced sequence of abelian groups
		\[
			0\longrightarrow\mathcal{C}(Q,A)\longrightarrow \mathcal{C}(Q,Q)\longrightarrow\mathcal{C}(Q,B)\longrightarrow 0
		\]
		is exact.  Dually, we have the definition of quasi-$\xi$-injective.  We denote by $\QP$ (resp. $\QI$) the  full subcategory  of quasi-$\xi$-projective (resp. quasi-$\xi$-injective) objects of $\Mcc$.
	\end{defn}

	By \citep[Remark 3.2]{He}, we know that $\P\subseteq \QP$ and $\QP$ is closed under isomorphisms. The following proposition gives a sufficient condition under which every quasi-$\xi$-projective object (resp. quasi-$\xi$-injective object) is $\xi$-projective  (resp. $\xi$-injective).
	\begin{prop}
		$(1)$ Assume that the direct sum of any quasi-$\xi$-projective object with any $\xi$-projective object is quasi-$\xi$-projective. Then every quasi-$\xi$-projective object is actually $\xi$-projective.

        $(2)$ Assume that the direct sum of any quasi-$\xi$-injective object with any $\xi$-injective object is quasi-$\xi$-injective. Then every quasi-$\xi$-injective object is actually $\xi$-injective.
	\end{prop}
    
	\begin{proof}
		$(1)$ Let $Q$ be a quasi-$\xi$-projective object. Since $\Mcc$ has enough $\xi$-projectives, there is an $\Mbe$-triangle $K\stackrel{x}\>P\stackrel{y}\>Q\stackrel{\delta}\dashrightarrow$ in $\xi$ with $P\in P(\xi)$. Observe that the canonical projection $[1_P \hspace{5pt} 0]:P\oplus Q\to P$ is a $\xi$-deflation. By \citep[Corollary 3.5]{JDP}, the composition
		\[[y \hspace{5pt} 0]=y\circ [1_P\hspace{5pt} 0]:P\oplus Q\to Q
		\]
		is also a $\xi$-deflation, which gives an $\Mbe$-triangle in $\xi$ of the form
		\[
			K'\longrightarrow P\oplus Q\stackrel{[y \hspace{5pt} 0]}{\longrightarrow} Q\dashrightarrow
		\]
		in $\xi$. Since $P\oplus Q$ is quasi-$\xi$-projective by assumption, there exists a morphism $g:P\oplus Q\rightarrow P\oplus Q$ making the diagram 
		\[
			\xymatrix{
				& &P\oplus Q\ar[d]^{[0 \hspace{5pt} 1_Q]}\ar@{..>}[dl]_{\exists g}\\
				K'\ar[r]&P\oplus Q\ar[r]_{[y \hspace{5pt} 0]}&Q\ar@{-->}[r] &
			}
		\]
		commute in $\Mcc$. Let $f=[1_P \hspace{5pt} 0]g\begin{small}\begin{bmatrix}0\\1_Q\end{bmatrix}\end{small}$ be a morphism in $\Mcc(Q,P)$. Note that
		\[
			yf=y[1_P \hspace{5pt} 0]g\begin{small}\begin{bmatrix}0\\1_Q\end{bmatrix}\end{small}=[y \hspace{5pt} 0]g\begin{small}\begin{bmatrix}0\\1_Q\end{bmatrix}\end{small}=[0 \hspace{5pt} 1_Q]\begin{small}\begin{bmatrix}0\\1_Q\end{bmatrix}\end{small}=1_Q.
		\]
		Thus, the $\Mbe$-triangle $K\stackrel{x}\>P\stackrel{y}\>Q\stackrel{\delta}\dashrightarrow$ is split by Lemma \ref{FJ}. Consequently, we have $Q\oplus K\cong P$, which implies $Q$ is $\xi$-projective.

        $(2)$ The proof is formally dual.
	\end{proof}

	The  quasi-$\xi$-projective dimension $\Qpd{A}$ of an object $A$ is defined inductively. When $A=0$, put $\Qpd{A}=-1$.  If $A\in\QP$, then define $\Qpd{A}=0$. Next by induction, for an integer $n>0$, put $\Qpd{A}\leqslant n$ if there exists an $\mathbb{E}$-triangle $K\rightarrow Q\rightarrow A\dashrightarrow$ in $\xi$ with $Q\in\QP$ and $\Qpd{K}\leqslant n-1$.

	We say that $\Qpd{A}=n$ if $\Qpd{A}\le n$ but $\Qpd{A}\not\le n-1$. If no nonnegative integer $n$ satisfies $\Qpd{A}=n$, we set $\Qpd{A}=\infty$. Dually we can define the quasi-$\xi$-injective dimension $\Qid{A}$ of an object $A\in\Mcc$.

	We use $\widehat{\mathcal{QP}}(\xi)$ (resp. $\widehat{\mathcal{QI}}(\xi)$) to denote the full subcategory of $\Mcc$ whose objects have finite quasi-$\xi$-projective (resp. quasi-$\xi$-injective) dimension.
\begin{defn}[{\citep[Definition 4.2]{He}}]
Let $\mathbf P$ be a $\Mcc(-,\QP)$-exact complex in $\Mcc$
\[
    \mathbf P:
    \xymatrix{
        \cdots\ar[r]
        &P_1\ar[r]^{d_1}
        &P_0\ar[r]^{d_0}
        &P_{-1}\ar[r]
        &\cdots
    }
\]
such that $P_n$ is $\xi$-projective for every $n\in\mathbb Z$.
We call such a complex $\mathbf P$ a complete
quasi-$\xi$-projective resolution. For each $n\in\mathbb Z$, there is a
$\Mcc(-,\QP)$-exact $\mathbb E$-triangle
\[
    K_{n+1}\stackrel{g_n}\longrightarrow
    P_n\stackrel{f_n}\longrightarrow
    K_n\stackrel{\delta_n}\dashrightarrow
\]
in $\xi$, which is the $\xi$-resolution $\mathbb E$-triangle
of $\mathbf P$. The object $K_n$ is called
quasi-$\xi$-Gorenstein projective
(quasi-$\xi$-$\mathcal G$projective for short).

Dually, we define quasi-$\xi$-Gorenstein injective
(quasi-$\xi$-$\mathcal G$injective for short) objects.
We denote by $\mathcal{Q}\Mcg\Mcp(\xi)$
(resp. $\mathcal{Q}\GI$) the full subcategory of
quasi-$\xi$-$\Mcg$projective
(resp. quasi-$\xi$-$\Mcg$injective) objects in $\Mcc$.
\end{defn}
    
\begin{rem} \label{KZB}
By definition, $\P\subseteq\QGP\subseteq\GP$, where $\GP$ denotes the
full subcategory of $\xi$-Gorenstein projective objects (see \citep[Definitions~4.7 and~4.8]{JDP}). The class $\QGP$ is closed under finite direct sums, direct summands,
and $\xi$-extensions
(see \citep[Theorems~4.7 and~4.8]{He}).

	\end{rem}

	We present some characterizations of  quasi-$\xi$-$\mathcal{G}$projective objects as follows.

    \begin{lem}
		Let $M$ be an object in $\Mcc$. The following statements are equivalent:

		(1) $M$ is a quasi-$\xi$-$\mathcal{G}$projective object.

		(2) There exist two $\xi$-exact complexes
		\[
			\cdots\>P_1\>P_0\>M\>0 \text{\, and \,} 0\>M\>P_{-1}\>P_{-2}\>\cdots
		\]
		with $P_n\in\P$ for any $n\in\mathbb{Z}$, which are both $\Mcc(-,\QP)$-exact.

		(3) There exists a $\Mcc(-,\QP)$-exact $\Mbe$-triangle $K_{n+1}\>P_n\>K_n\dashrightarrow$ in $\xi$ with $P_n\in\P$ for any $n\in\mathbb{Z}$ and $K_0=M$.
	\end{lem}
	\begin{proof}
		It follows from the definition of quasi-$\xi$-$\mathcal{G}$projective object.
	\end{proof}

	\begin{lem}\label{HH}
		Let $A\longrightarrow B \longrightarrow C  \dashrightarrow$  be an $\mathbb{E}$-triangle in $\xi$ with $C\in\QGP$.  Then we have the following statements.

		(1) $A\in\QGP$ if and only if $B\in\QGP$.

        (2) The $\Mbe$-triangle $A\>B\>C\dashrightarrow$ is $\Mcc(-,\widehat{\mathcal{QP}}(\xi))$-exact.

	\end{lem}

	\begin{proof}
        (1)  follows from  \citep[Theorem 4.7]{He}.
    
		(2) Since $C\in\QGP$,  there is  an $\Mbe$-triangle $K\stackrel{g}\>P\stackrel{f}\>C\stackrel{\delta}\dashrightarrow$ in $\xi$ with $P\in\P$ and $K\in\QGP$. Now we claim that  for any $M\in\Mcc$ with $\Qpd{M}=n$, the sequence of abelian groups
		\[
			0\>\Mcc(C,M)\stackrel{\Mcc(f,M)}\>\Mcc(P,M)\stackrel{\Mcc(g,M)}\>\Mcc(K,M)\>0\tag{$*$}
		\]
		is exact. If $n=0$, then the sequence $(*)$ is exact
		by  \citep[Lemma 4.4(1)]{He}. Assume that the sequence $(*)$ is exact for any $M\in\Mcc$ with $\Qpd{M}=n-1$. Next we consider the case of $\Qpd{M}=n$. Suppose that $L\stackrel{x}\>Q\stackrel{y}\>M\stackrel{\rho}\dashrightarrow$ is an $\Mbe$-triangle in $\xi$ with $Q\in\QP$ and $\Qpd{L}=n-1$. By induction, the sequence of abelian groups
		\[
			0\>\Mcc(C,L)\stackrel{\Mcc(f,L)}\>\Mcc(P,L)\stackrel{\Mcc(g,L)}\>\Mcc(K,L)\>0\tag{$*$}
		\]
		is exact. Note that the functor $\Mcc(-,-)$ is biadditive, we have the following commutative diagram.
		\[
			\xymatrix{
				&0\ar[d]&0\ar[d]&0\ar@{-->}[d]\\
				0\ar@{-->}[r]&\mathcal{C}(C,L)\ar[d]_{\Mcc(f,L)}\ar[r]^{\Mcc(C,x)}&\mathcal{C}(C,Q)\ar[d]_{\Mcc(f,Q)}\ar[r]^{\Mcc(C,y)} & \Mcc(C,M)\ar[d]^{\Mcc(f,M)}\ar@{-->}[r]&0\\
				0\ar[r]&\mathcal{C}(P,L) \ar[d]_{\Mcc(g,L)}\ar[r]^{\Mcc(P,x)} & \mathcal{C}(P,Q) \ar[r]^{\Mcc(P,y)}\ar[d]_{\Mcc(g,Q)}& \mathcal{C}(P,M) \ar[d]^{\Mcc(g,M)} \ar[r]&0 \\
				0\ar[r]&\mathcal{C}(K,L)\ar[d]\ar[r]^{\Mcc(K,x)} & \mathcal{C}(K,Q)\ar[d] \ar[r]^{\Mcc(K,y)} & \mathcal{C}(K,M)  \ar@{-->}[d] \ar@{-->}[r]&0\\
				&0&0&0
			}
		\]
		It is easy to see that the morphism $\Mcc(C,x)$ is monic. For any morphism $h:C\to M$, there is a morphism $h_1: P\to Q$ such that $hf=yh_1$. It follows from $\rm(ET3)^{op}$ that there exists a morphism $h_2$, which gives the following commutative diagram.
		\[
			\xymatrix{
				K \ar@{-->}[d]^{h_2} \ar[r]^{g} & P \ar[d]^{h_1} \ar[r]^{f} & C\ar[d]^{h}\ar@{-->}[r]^{\delta}& \\
				L\ar[r]^{x} &Q\ar[r]^{y} & M\ar@{-->}[r]^{\rho}&}
		\]
		Since $\Mcc(g,L):\Mcc(P,L)\to\Mcc(K,L)$ is epic, $h_2$ factors through $g$. Hence $h$ factors through $y$ by Lemma \ref{FJ}. So $\Mcc(C,y)$ is epic, which implies the exactness of the sequence of abelian groups
		\[
			0\>\Mcc(C,L)\stackrel{\Mcc(C,x)}\>\Mcc(C,Q)\stackrel{\Mcc(C,y)}\>\Mcc(C,M)\>0.
		\]
		Similarly, we can obtain that the sequence of abelian groups
		\[
			0\>\Mcc(K,L)\stackrel{\Mcc(K,x)}\>\Mcc(K,Q)\stackrel{\Mcc(K,y)}\>\Mcc(K,M)\>0
		\]
		is exact since $K\in\QGP$. Thus, the sequence $(*)$ is exact by $3\times 3$ Lemma. By
		Lemma \ref{BH}(1), we have the following commutative diagram.
		\[
			\xymatrix{
				&K\ar[d]\ar@{=}[r]&K\ar[d]\\
				A\ar@{=}[d]\ar[r]&X\ar[r]\ar[d]&P\ar[d]\ar@{-->}[r]&\\
				A\ar[r]&B\ar@{-->}[d]\ar[r]&C\ar@{-->}[d]&\\
				&&&&}
		\]
		Note that $A\>X\>P\dashrightarrow$ is a split $\Mbe$-triangle. So the  sequence of abelian groups
		\[
			0\>\Mcc(C,M)\>\Mcc(B,M)\>\Mcc(A,M)\>0
		\]
		is exact by  \citep[Lemma 4.10(1)]{JDP}.
	\end{proof}

\begin{prop}\label{qgpqpcap}
    We have $\QGP\cap\QP=\P$.
    Consequently, $\P\subsetneq\QGP$ if and only if $\QGP\nsubseteq\QP$.
\end{prop}

\begin{proof}
    Obviously, $\P\subseteq\QGP\cap\QP$.
    Conversely, let $M\in\QGP\cap\QP$. Since $M\in\QGP$,
    there is an $\mathbb E$-triangle
    \[
        M\xrightarrow{u}P\xrightarrow{v}K\dashrightarrow
    \]
    in $\xi$ with $P\in\P$ and $K\in\QGP$, which is
    $\Mcc(-,\QP)$-exact by definition. Since $M\in\QP$,
    we have the following exact sequence
    \[
        0\longrightarrow\Mcc(K,M)
        \longrightarrow\Mcc(P,M)
        \xrightarrow{\Mcc(u,M)}\Mcc(M,M)
        \longrightarrow0.
    \]
    Hence the $\mathbb E$-triangle
    $M\xrightarrow{u}P\xrightarrow{v}K\dashrightarrow$
    is split. So we obtain $M\oplus K\cong P$.
    Since $\P$ is closed under direct summands, it follows
    that $M\in\P$. Thus, $\QGP\cap\QP=\P$.

    The last assertion follows from the above equality
    and $\P\subseteq\QGP$.
\end{proof}

	\begin{prop}\label{TTTT}
		$(1)$ Assume that $G$ is an object  in $\QGP$. If $M\in\widehat{\mc{QP}}(\xi)$, then $\ext^0(G,M) \simeq \Mcc(G,M)$ and $\ext^n(G,M)=0$ for any  $n\geq 1$.

        $(2)$ Assume that $H$ is an object  in $\QGI$. If $M\in\widehat{\mathcal{QI}}(\xi)$, then $\ext^0(M,H) \simeq \Mcc(M,H)$ and $\ext^n(M,H)=0$ for any  $n\geq 1$.
	\end{prop}
	\begin{proof}
		(1) Since $G\in\QGP$, there exists an $\Mbe$-triangle $G'\longrightarrow P \longrightarrow G  \dashrightarrow$ in $\xi$ with $G'\in\QGP$ and $P\in\P$. If $M\in\widehat{\mc{QP}}(\xi)$, then we have the following commutative diagram
		\[
			\xymatrix{
				0\ar[r]&\Mcc(G,M)\ar[r]\ar[d]_{\varphi_1}&\Mcc(P,M)\ar[r]\ar[d]^{\varphi_2}_{\simeq}&\Mcc(G',M)\ar[r]\ar[d]^{\varphi_3}& 0&\\
				0\ar[r]&\ext^0(G,M)\ar[r]&\ext^0(P,M)\ar[r]&\ext^0(G',M)\ar[r]&\ext^1(G,M)\ar[r]&0.
			}
		\]
		where the top exact sequence follows from Lemma \ref{HH}(2). Note that  $\varphi_1$ and  $\varphi_3$ are monic, hence $\varphi_1$ is  epic by the snake Lemma, so $\varphi_1$ is an isomorphism. Similarly, one can get that $\varphi_3$ is an isomorphism, so $\ext^1(G,M)=0$. It is easy to show that $\ext^n(G, M)=0$ for any $n\geq 1$ by Lemma \ref{LZHL}.
        
        $(2)$ is similar to $(1)$.
	\end{proof}

\begin{exam}\label{separatingexample}
Let $k$ be a field, $A=k[\varepsilon]/(\varepsilon^2)$
and $S=A/(\varepsilon)$. Set
$\mathcal C=K^b(\operatorname{mod}A)$, where
$\operatorname{mod}A$ denotes the category of finitely
generated $A$-modules. Let $\xi$ consist of the
distinguished triangles
$X\to Y\to Z\xrightarrow{\delta}X[1]$ with
$H^n(\delta)=0$ for every $n\in\mathbb Z$. Then
\[
    \mathcal P(\xi)\subsetneq\mathcal{QGP}(\xi)
    \subsetneq\mathcal{GP}(\xi)=\mathcal C.
\]
Moreover, $\mathcal{QP}(\xi)$ and $\mathcal{QGP}(\xi)$
are incomparable, and the induced relative
extriangulated structure is neither exact nor triangulated.
\end{exam}

\begin{proof}
The class $\xi$ contains the split triangles and is
closed under isomorphisms and finite direct sums.
For a triangle
$X\longrightarrow Y\longrightarrow Z
\xrightarrow{\delta}X[1]$ in $\xi$ and morphisms
$c:Z'\to Z$ and $a:X\to X'$, the base and cobase
changes have connecting morphisms $\delta c$ and
$a[1]\delta$, respectively. Since $H^n(\delta)=0$
for any $n\in\mathbb Z$, both connecting morphisms
induce zero in cohomology. Hence $\xi$ is closed
under base and cobase change.

In the diagram of Lemma~\ref{BH}\textup{(1)}, let
$\delta_1:C\to A_1[1]$ be the connecting morphism
of the bottom row. Suppose that the middle row
and the right column belong to $\xi$.
The middle row has connecting morphism $\delta_1y_2$.
For any $n\in\mathbb Z$, the long exact cohomology
sequence of the right column implies that
$H^n(y_2)$ is epic. Since the middle row belongs
to $\xi$, we have
$H^n(\delta_1)H^n(y_2)=H^n(\delta_1y_2)=0$.
Hence $H^n(\delta_1)=0$, so the bottom row belongs
to $\xi$. Thus $\xi$ is saturated and therefore
proper. By Lemma~\ref{Thm3.2},
$(\mathcal C,\mathbb E_\xi,\mathfrak s_\xi)$ is
extriangulated. Its inflations and deflations are
precisely the cohomologically injective and
surjective maps, respectively, so
Condition~\ref{WIC} holds.

Since $A$ is self-injective, we have natural
isomorphisms
\[
\mathcal C(A[-n],X)\cong H^n(X)\; \text{and}\;\mathcal C(X,A[-n])
\cong\operatorname{Hom}_A(H^n(X),A)\]
Thus each $A[n]$ is both $\xi$-projective and
$\xi$-injective.
For any $X\in\mathcal C$, finite free surjections
$A^{m_n}\twoheadrightarrow H^n(X)$ and embeddings
$H^n(X)\hookrightarrow A^{r_n}$ induce a
$\xi$-deflation $\bigoplus_n A[-n]^{m_n}\to X$
and a $\xi$-inflation
$X\to\bigoplus_n A[-n]^{r_n}$, respectively.
Both sums are finite. It follows that $\mathcal C$
has enough $\xi$-projectives and $\xi$-injectives,
and that
$\mathcal P(\xi)=\mathcal I(\xi)
=\operatorname{add}\{A[n]\mid n\in\mathbb Z\}$.
Since $\mathcal P(\xi)=\mathcal I(\xi)$, splicing
a $\xi$-projective resolution and a $\xi$-injective
coresolution of any object gives a complete
$\xi$-projective resolution. Hence
$\mathcal{GP}(\xi)=\mathcal C$.

Consider the complexes
\[
    E=(S\xrightarrow{\iota}A\xrightarrow{\pi}S),
    \qquad B=(A\xrightarrow{\varepsilon}A),
\]
in degrees $0,1,2$ and $0,1$, respectively, where
$\iota(1)=\varepsilon$ and $\pi$ is the canonical map.
Let $Q\in\mathcal{QP}(\xi)$ and let $N\in\mathcal C$
be acyclic. The triangle
$N[-1]\to Q\xrightarrow{(1,0)}Q\oplus N\to N$
lies in $\xi$, so quasi-projectivity implies
$\mathcal C(Q,N)=0$.

Let $Q$ be represented by a complex concentrated
in degrees $[a,b]$. There exists a quasi-isomorphism
$P\to Q$, where $P$ is a bounded-above complex of
finitely generated projective $A$-modules.
The smart truncation $Z$ of $P$ in degree $a-1$
is bounded, has projective terms in degrees at
least $a$, and admits an induced quasi-isomorphism
$s:Z\to Q$.
Since $\operatorname{Cone}s$ is acyclic, we have
$\mathcal C(Q,\operatorname{Cone}s)=0$.
Hence $s$ has a section in $\mathcal C$.
A chain-map representative of this section factors
through the brutal truncation $Z^{\geq a}$.
Thus $Q$ is a direct summand in $\mathcal C$ of
a bounded complex of projective-injective
$A$-modules. Maps from $N$ into such a complex
are null-homotopic, so $\mathcal C(N,Q)=0$.

The complex $E$ is acyclic and nonzero in
$\mathcal C$, since the short exact sequence
defining it does not split.
For any $i\in\mathbb Z$, the triangle
$E[-i-1]\to0\to E[-i]\xrightarrow{1}E[-i]$
belongs to $\xi$ and is
$\mathcal C(-,\mathcal{QP}(\xi))$-exact.
These are the resolution triangles of a complete
quasi-$\xi$-projective resolution with zero terms.
Hence $E\in\mathcal{QGP}(\xi)$.
If $E\in\mathcal{QP}(\xi)$, then
$\mathcal C(E,E)=0$, contrary to $E\neq0$.
Thus $E\notin\mathcal{QP}(\xi)$, and hence
$E\notin\mathcal P(\xi)$.

For any $Y\in\mathcal C$ with
$\varepsilon H^n(Y)=0$ for all $n\in\mathbb Z$,
the distinguished triangle
$B\to A[0]\xrightarrow{\varepsilon}A[0]\to B[1]$
induces a natural short exact sequence
\[
    0\longrightarrow H^0(Y)
    \longrightarrow\mathcal C(B,Y)
    \longrightarrow H^1(Y)\longrightarrow0.
\]
Let $X\to B\xrightarrow{q}Y\to X[1]$ be a triangle
in $\xi$. Note that $H^0(B)\cong H^1(B)\cong S$
and all other cohomology groups of $B$ vanish.
Since $H^n(q)$ is epic for any $n\in\mathbb Z$,
we have $\varepsilon H^n(Y)=0$ for all
$n\in\mathbb Z$.
Thus $B$, $Y$ and their shifts satisfy the
preceding condition.
Naturality and the surjectivity of $H^n(q)$ imply
that $\mathcal C(B,q)$ and
$\mathcal C(B,q[-1])$ are surjective.
The long exact Hom sequence therefore restricts
to a short exact sequence
$0\to\mathcal C(B,X)\to\mathcal C(B,B)
\to\mathcal C(B,Y)\to0$.
Hence $B\in\mathcal{QP}(\xi)$.
Since $H^0(B)\cong S$ is not a projective
$A$-module, $B\notin\mathcal P(\xi)$.
By Proposition~\ref{qgpqpcap},
$B\notin\mathcal{QGP}(\xi)$.
This proves the strict inclusions and incomparability.

Since $E$ is acyclic, the triangle
$E\to0\to E[1]\xrightarrow{1}E[1]$ belongs to $\xi$.
Thus $E\to0$ is a $\xi$-inflation, but it is not
monic since $E\neq0$.
As inflations in an exact category are monic,
$(\mathcal C,\mathbb E_\xi,\mathfrak s_\xi)$ is
not exact.
On the other hand, $H^0(A[0])=A\neq0$, so the
morphism $A[0]\to0$ is not cohomologically injective
and hence is not a $\xi$-inflation.
Since any morphism in a triangulated category is
an inflation,
$(\mathcal C,\mathbb E_\xi,\mathfrak s_\xi)$ is
not triangulated.
\end{proof}

	\begin{prop}\label{Prop3.7}
		Let $A\longrightarrow B\longrightarrow C \dashrightarrow$ be an $\mathbb{E}$-triangle in $\xi$ with both $A,B\in\QGP$. Then the following statements are equivalent.

		(1) $C\in\QGP$.

		(2) $\ext^1(C,Q)=0$, $\forall Q\in\QP$.

		(3) $\ext^1(C,P)=0$, $\forall P\in\P$.
	\end{prop}

	\begin{proof}
		(1) $\Rightarrow$ (2) It follows from Proposition \ref{TTTT}(1).

		(2) $\Rightarrow$ (3) It is trivial since $\P\subseteq\QP$.

		(3) $\Rightarrow$ (1) Since $A\in\QGP$, there is an $\mathbb{E}$-triangle $A\longrightarrow P \longrightarrow K\dashrightarrow$  with $P\in\P$ and $K\in\QGP$. Then we have the following commutative diagram
		\[
			\xymatrix{
				A\ar[d]\ar[r]&B\ar[d]\ar[r]&C\ar@{=}[d]\ar@{-->}[r] &\\
				P\ar[d]\ar[r]&G\ar[r]\ar[d]&C\ar@{-->}[r] &\\
				K\ar@{=}[r]\ar@{-->}[d]&K\ar@{-->}[d]\\
				&&}
		\]
		where all rows and columns are $\mathbb{E}$-triangles in $\xi$ by Lemma \ref{BH}(2) and Lemma \ref{Thm3.2}. So we obtain $G\in\QGP$ by Remark \ref{KZB}. For the $\mathbb{E}$-triangle $P\longrightarrow G \longrightarrow C\dashrightarrow$, we have the following commutative diagram by hypothesis
		\[
			\xymatrix{
				&\Mcc(C,P)\ar[r]\ar[d]&\Mcc(G,P)\ar[r]\ar[d]^{\varphi}_{\simeq}&\Mcc(P,P)\ar@{..>}[r]\ar[d]_{\simeq}^{\psi} & 0&\\
				0\ar[r]&\ext^0(C,P)\ar[r]&\ext^0(G,P)\ar[r]&\ext^0(P,P)\ar[r]&\ext^1(C,P)=0.
			}
		\]
		where $\varphi$ and $\psi$ are isomorphisms by Proposition \ref{TTTT}(1) since $P\in\P\subseteq \QGP$. Hence, the $\mathbb{E}$-triangle $P\longrightarrow G \longrightarrow C\dashrightarrow$ is split, and hence $C\in\QGP$ by \cite[Theorem 4.8]{He}.
	\end{proof}

	\begin{lem}[{\citep[Lemma 4.6]{He}}]\label{Dec30}
		(1) If $A\stackrel{f}\longrightarrow B \stackrel{f'}\longrightarrow C  \stackrel{\delta}\dashrightarrow$ and $B\stackrel{g}\longrightarrow D\stackrel{g'}\longrightarrow E  \stackrel{\delta'}\dashrightarrow$ are both $\Mcc(-,\QP)$-exact   $\mathbb{E}$-triangles in $\xi$, then we have the following commutative diagram:
		\[
			\xymatrix{
				A\ar@{=}[d] \ar[r]^{f} & B  \ar[r]^{f'}\ar[d]^{g} & C \ar[d]^d \ar@{-->}[r]^{\delta} &  \\
				A\ar[r]^{h} & D\ar[d]^{g'} \ar[r]^{h'} & F\ar[d]^e \ar@{-->}[r]^{\delta''} & \\
				& E\ar@{=}[r]\ar@{-->}[d]^{\delta'}&E\ar@{-->}[d]^{f'_*\delta'} \\
				& & &}
		\]
		where all rows and columns are $\Mcc(-,\QP)$-exact  $\mathbb{E}$-triangles in $\xi$.

		(2) If $A\stackrel{f}\longrightarrow B \stackrel{f'}\longrightarrow C  \stackrel{\delta}\dashrightarrow$ and $D\stackrel{g}\longrightarrow C\stackrel{g'}\longrightarrow E  \stackrel{\delta'}\dashrightarrow$ are both $\Mcc(-,\QP)$-exact   $\mathbb{E}$-triangles in $\xi$, then we have the following commutative diagram:
		\[
			\xymatrix{
				A\ar@{=}[d] \ar[r]^{d} & F  \ar[r]^{e}\ar[d]^{h} & D \ar[d]^{g} \ar@{-->}[r]^{g^*\delta} &  \\
				A\ar[r]^f & B\ar[d]^{h'} \ar[r]^{f'} & C\ar[d]^{g'} \ar@{-->}[r]^{\delta} & \\
				& E\ar@{=}[r]\ar@{-->}[d]^{\delta''}&E\ar@{-->}[d]^{\delta'} \\
				& & &}
		\]
		where all rows and columns are $\Mcc(-,\QP)$-exact  $\mathbb{E}$-triangles in $\xi$.
	\end{lem}

	\begin{prop}\label{prop:qgp-resolution-replacement}
		$(1)$ Let $M$ be an object in $\Mcc$. Assume that there is a $\xi$-exact complex
		\[
			\cdots\>G_n\>\cdots\>G_1\>G_0\>M\>0 \tag{3.1}
		\]
		with all $G_n\in\QGP$. Then there is a $\xi$-exact complex
		\[
			\cdots\>P_n\>\cdots\>P_1\>P_0\>M\>0 \tag{3.2}
		\]
		with all $P_n\in\P$. Moreover, if the $\xi$-exact complex (3.1) is $\Mcc(-,\QP)$-exact, then so is (3.2).

        $(2)$ Let $M$ be an object in $\Mcc$. Assume that there is a $\xi$-exact complex
		\[
			0\>M\>G_{-1}\>G_{-2}\>\cdots\>G_{-n}\>\cdots\tag{3.3}
		\]
		with all $G_{-n}\in\QGP$. Then there is a $\xi$-exact complex
		\[
			0\>M\>P_{-1}\>P_{-2}\>\cdots\>P_{-n}\>\cdots\tag{3.4}
		\]
		with all $P_{-n}\in\P$. Moreover, if the $\xi$-exact complex (3.3) is $\Mcc(-,\QP)$-exact, then so is (3.4).
	\end{prop}
	\begin{proof}
        We only prove $(1)$. The proof of $(2)$ is similar.

		$(1)$ By (3.1), we have the following $\Mbe$-triangles
		\[
			K_1\>G_0\>M\dashrightarrow\text{\, and \,}  K_2\>G_1\>K_1\dashrightarrow
		\]
		in $\xi$. Since $G_0\in\QGP$, there exists a $\Mcc(-,\QP)$-exact $\Mbe$-triangle $G\>P_0\>G_0\dashrightarrow$ in $\xi$ with $P_0\in\P$ and $G\in\QGP$. By Lemma \ref{Thm3.2} and $\rm(ET4)^{op}$, we have the following commutative diagram
		\[
			\xymatrix{
				G\ar@{=}[d] \ar[r] & L \ar[r]\ar[d] & K_1 \ar[d] \ar@{-->}[r] &  \\
				G\ar[r] & P_0\ar[d] \ar[r] & G_0\ar[d] \ar@{-->}[r]& \\
				& M\ar@{=}[r]\ar@{-->}[d]&M\ar@{-->}[d] \\
				& & &}\tag{$*$}
		\]
		where all  rows and columns are  $\mathbb{E}$-triangles in $\xi$. By Lemma \ref{BH}(1) and Lemma \ref{Thm3.2}, we obtain the following commutative diagram
		\[
			\xymatrix{
				&G\ar@{=}[r]\ar[d]&G\ar[d]\\
				K_2\ar[r]\ar@{=}[d]&G'\ar[r]\ar[d]&L\ar[d]\ar@{-->}[r]&\\
				K_2\ar[r]&G_1\ar[r]\ar@{-->}[d]&K_1\ar@{-->}[r]\ar@{-->}[d]&\\
				&&&
			}\tag{$**$}
		\]
		whose rows and columns are $\mathbb{E}$-triangles in $\xi$. It follows from  Lemma \ref{HH}(1) that  $G'\in\QGP$. Repeating this process, we can obtain a $\xi$-exact complex
		\[
			\cdots\>P_n\>\cdots\>P_1\>P_0\>M\>0
		\]
		with all $P_n\in\P$. Moreover, if the complex (3.1) is $\Mcc(-,\QP)$-exact, then the $\Mbe$-triangle $K_1\>G_0\>M\dashrightarrow$ is $\Mcc(-,\QP)$-exact. By applying Lemma \ref{Dec30}(2) to diagram $(*)$, we conclude that the $\Mbe$-triangle $G\>L\>K_1\dashrightarrow$ is $\Mcc(-,\QP)$-exact. For any $Q\in\QP$, applying the functor $\Mcc(-,Q)$ to the diagram ($**$), one obtains the following commutative diagram of abelian groups by Lemma \ref{HH}(2):
		\[
			\xymatrix{
				&0\ar[d]&0\ar[d]&0\ar[d]\\
				0\ar[r]&\mathcal{C}(K_1,Q)\ar[d]\ar[r]&\mathcal{C}(G_1,Q)\ar[d]\ar[r] & \Mcc(K_2,Q)\ar@{=}[d]\ar[r]&0\\
				0\ar@{..>}[r]&\mathcal{C}(L,Q) \ar[d]\ar[r]& \mathcal{C}(G',Q)  \ar[r]\ar[d]& \mathcal{C}(K_2,Q) \ar[d] \ar@{..>}[r]&0  \\
				0\ar[r]&\mathcal{C}(G,Q)\ar[d]\ar@{=}[r] & \mathcal{C}(G,Q)\ar[d] \ar[r] & 0 \ar[d] \ar[r]&0\\
				&0&0&0
			}
		\]
		So the $\Mbe$-triangle $K_2\>G'\>L\dashrightarrow$ is $\Mcc(-,Q)$-exact by $3\times 3$ Lemma. Repeating this process, we obtain that the $\xi$-exact complex (3.2) is $\Mcc(-,\QP)$-exact.
	\end{proof}

	\begin{cor}\label{Cor3.8}
		$(1)$ Let $M$ be an object in $\Mcc$. Then $M$ admits a $\xi$-exact complex
		\[
			\cdots\>G_n\>\cdots\>G_1\>G_0\>M\>0
		\]
		with all $G_n\in\QGP$, which is $\Mcc(-,\QP)$-exact if and only if $M$ admits a $\xi$-exact complex
		\[
			\cdots\>P_n\>\cdots\>P_1\>P_0\>M\>0
		\]
		with all $P_n\in\P$, which is $\Mcc(-,\QP)$-exact.

        $(2)$ Let $M$ be an object in $\Mcc$. Then $M$ admits a $\xi$-exact complex
		\[
			0\>M\>G_{-1}\>G_{-2}\>\cdots\>G_{-n}\>\cdots
		\]
		with all $G_{-n}\in\QGP$, which is $\Mcc(-,\QP)$-exact if and only if $M$ admits a  $\xi$-exact complex
		\[
			0\>M\>P_{-1}\>P_{-2}\>\cdots\>P_{-n}\>\cdots
		\]
		with all $P_{-n}\in\P$, which is $\Mcc(-,\QP)$-exact.
	\end{cor}

	\begin{thm}\label{Thm3.12}
		Let $M$ be an object in $\Mcc$. Then the following statements are equivalent:

		(1) $M$ is a quasi-$\xi$-$\mathcal{G}$projective object.

		(2)  There exist $\Mcc(-,\QP)$-exact and $\Mcc(\QP,-)$-exact  $\Mbe$-triangles $K_{n+1}\>G_n\>K_n\dashrightarrow$ in $\xi$ such that $G_n\in\QGP$ for any $n\in\mathbb{Z}$ and $K_0=M$.

		(3) There exist $\Mcc(-,\QP)$-exact $\Mbe$-triangles $K_{n+1}\>G_n\>K_n\dashrightarrow$ in $\xi$ such that $G_n\in\QGP$ for any $n\in\mathbb{Z}$ and $K_0=M$.

		(4) There exists a $\Mcc(-,\QP)$-exact and $\Mcc(\QP,-)$-exact $\Mbe$-triangle $M\>G\>M\dashrightarrow$ in $\xi$ with $G\in\QGP$.

		(5) There exists a $\Mcc(-,\QP)$-exact $\Mbe$-triangle $M\>G\>M\dashrightarrow$ in $\xi$ with $G\in\QGP$.
	\end{thm}
	\begin{proof}
		(1) $\Rightarrow$ (2) Let $M$ be a quasi-$\xi$-$\mathcal{G}$projective object of $\Mcc$. Consider the split $\Mbe$-triangles
		\[
			0\>M\stackrel{1}\> M\dashrightarrow\text{\, and \,} M\stackrel{1}\>M\> 0\dashrightarrow
		\]
		in $\xi$, which are  $\Mcc(-,\QP)$-exact and $\Mcc(\QP,-)$-exact.

		(2) $\Rightarrow$ (3) It is clear.

		(1) $\Rightarrow$ (4)  Consider the following split $\Mbe$-triangle
		\[
			M \stackrel{\begin{tiny}\begin{bmatrix}
						1 \\
						0
					\end{bmatrix}\end{tiny}}\>M \oplus M \stackrel{[0 \hspace{5pt} 1]}{\longrightarrow} M\dashrightarrow
		\]
		in $\xi$, which is $\Mcc(-,\QP)$-exact and $\Mcc(\QP,-)$-exact. Since $M\in\QGP$, we have $M\oplus M\in\QGP$ by Remark \ref{KZB}.

		(4) $\Rightarrow$ (5) $\Rightarrow$ (3)  It is clear.

		(3) $\Rightarrow$ (1) It follows from Corollary \ref{Cor3.8}.
	\end{proof}

	Set $\left[\QGP\right]^1=\QGP$. Inductively define the following subcategories of $\Mcc$:
	\[
		\begin{split}
			\left[\QGP\right]^{n+1}=&\left\{M \in \Mcc \mid\text{there exists a $\Mcc(-,\QP)$-exact and $\xi$-exact complex}\right. \\
			&\cdots \> G_1 \> G_0 \> G_{-1} \> G_{-2} \>\cdots \text{\, in \,}\Mcc\\
			&\text{\, with all \,} G_i\in[\QGP]^{n} \text{\, and a $\xi$-resolution $\Mbe$-triangle \,} K_1\>G_0\>M\dashrightarrow\left.\right\}.
		\end{split}
	\]
    Dually, we can define the subcategories $\left[\QGI\right]^n$ of $\Mcc$ for any $n\geq 1$.

	\begin{thm}\label{nice}
		For any $n\geq 1$, $\left[\QGP\right]^n=\QGP$ and $\left[\QGI\right]^n=\QGI$.
	\end{thm}
	\begin{proof}
		We only prove the first equality, the second equality is similar. It is easy to see that
		\[
			\QGP\subseteq \left[\QGP\right]^2\subseteq \left[\QGP\right]^3 \subseteq\cdots
		\]
		is an ascending chain of subcategories of $\Mcc$. By Theorem \ref{Thm3.12} (1) $\Leftrightarrow$ (3), we have that \[\left[\QGP\right]^2=\QGP.\] The assertion follows by  induction on $n$.
	\end{proof}

	\section{Relative quasi-$\xi$-Gorenstein homological dimensions}
	\quad~ In this section, we study the properties of relative homological dimensions for the subcategory $\QGP$ (resp. $\QGI$) in an extriangulated category $\Mcc$. Furthermore, we give equivalent characterizations of objects with finite quasi-$\xi$-$\mathcal{G}$projective (resp. quasi-$\xi$-$\mathcal{G}$injective) dimension.

	The quasi-$\xi$-$\mathcal{G}$projective dimension $\Qgpd{M}$ of an object $M$ is defined inductively. When $M=0$, put $\Qgpd{M}=-1$.  If $M\in\QGP$, then define $\Qgpd{M}=0$. By induction, for an integer $n>0$, put $\Qgpd{M}\leqslant n$ if there exists an $\mathbb{E}$-triangle $K\rightarrow G\rightarrow M\dashrightarrow$ in $\xi$ with $G\in \QGP$ and $\Qgpd{K}\leqslant n-1$.

	We say that $\Qgpd{A}=n$ if $\Qgpd{A}\le n$ but $\Qgpd{A}\not\le n-1$. If no nonnegative integer $n$ satisfies $\Qgpd{A}=n$, we set $\Qgpd{A}=\infty$.  Dually, we can define  the quasi-$\xi$-$\mathcal{G}$injective dimension $\Qgid{M}$ of any object $M$ in $\Mcc$.

	We use $\widehat{\mathcal{QGP}}(\xi)$ (resp.  $\widehat{\mathcal{QGI}}(\xi)$)  to denote the full subcategory of $\Mcc$ whose objects have finite  quasi-$\xi$-$\mathcal{G}$projective (resp.  quasi-$\xi$-$\mathcal{G}$injective) dimension.

	\begin{prop}\label{Prop4.4}
		For an integer $n\geq0$, let $M$ be an object in $\Mcc$ with $\Qgpd{M}\leq n$. Then there exists an $\Mbe$-triangle
		\[
			K\>G\>M\dashrightarrow
		\]
		in $\xi$ with $G\in\QGP$ and $\pd{K}\leq n-1$.
	\end{prop}
	\begin{proof}
The case $n=0$ follows from the split $\mathbb E$-triangle
$0\longrightarrow M\xrightarrow{1_M}M\dashrightarrow$,
since $\pd{0}=-1$. We may therefore assume $n\geq1$. Since $\Qgpd{M}\leq n$, there is a $\xi$-exact complex
		\[
			0\>P_n\>P_{n-1}\>\cdots\>P_2\>P_1\>G\>M\>0
		\]
		in $\Mcc$ with all $P_i\in\P$ and $G\in\QGP$ by \cite[Theorem 5.3]{He}, which implies that there is an $\Mbe$-triangle $K\>G\>M\dashrightarrow$ in $\xi$ with $G\in\QGP$ and $\pd{K}\leq n-1$.
	\end{proof}

	\begin{lem}\label{GFC}
		Let $M$ be an object in $\Mcc$ with $\Qgpd{M}\leq n<\infty$. Then there exists an $\Mbe$-triangle
		\[
			M\>L\>G^{\prime}\dashrightarrow
		\]
		in $\xi$ with $G^{\prime}\in\QGP$ and $\pd{L}\leq n$.
	\end{lem}
	\begin{proof}
		By Proposition \ref{Prop4.4}, the condition $\Qgpd{M}\leq n$ yields an $\Mbe$-triangle $K\>G\>M\dashrightarrow$ in $\xi$ with $G\in\QGP$ and $\pd{K}\leq n-1$. Since $G\in\QGP$, there also exists an $\Mbe$-triangle $G\>P\>G'\dashrightarrow$ in $\xi$ with $P\in\P$ and $G'\in\QGP$. Hence we have the following commutative diagram
		\[
			\xymatrix{
				K\ar@{=}[d] \ar[r] & G \ar[r]\ar[d] & M \ar[d] \ar@{-->}[r] &  \\
				K\ar[r] & P\ar[d] \ar[r] & L\ar[d] \ar@{-->}[r]& \\
				& G'\ar@{=}[r]\ar@{-->}[d]&G'\ar@{-->}[d] \\
				& & &}
		\]
		where all rows and columns are $\Mbe$-triangles in $\xi$. It is clear that $\pd{L}\leq \pd{K}+1\leq n$, and the third column in the above diagram is the required $\Mbe$-triangle.
	\end{proof}

	\begin{lem}[cf.\ {\citep[Proposition~5.4]{He}}]\label{lem4.4}
	Let $A\longrightarrow B\longrightarrow C\dashrightarrow$ be an
$\mathbb E$-triangle in $\xi$ with $C\in\mathcal{QGP}(\xi)$.
Then, for any integer $n\geq0$,
$\Qgpd{A}\leq n$ if and only if
$\Qgpd{B}\leq n$.
In particular, if $A$ and $B$ are nonzero, then
$\Qgpd{A}=\Qgpd{B}$.
	\end{lem}

 \begin{proof}
We proceed by induction on $n$. The case $n=0$ follows from
Lemma~\ref{HH}(1). Let $n>0$ and assume that the assertion
holds for $n-1$.

Suppose that $\xi\text{-}\operatorname{QGpd}(A)\leq n$.
There exist $\mathbb E$-triangles
$A_1\longrightarrow G_A\longrightarrow A\dashrightarrow$
and $C_1\longrightarrow P_C\longrightarrow C\dashrightarrow$
in $\xi$ with $G_A,C_1\in\mathcal{QGP}(\xi)$,
$P_C\in\mathcal P(\xi)$ and
$\Qgpd{A_1}\leq n-1$.
By Lemma \ref{BH}(1) and Lemma \ref{Thm3.2}, we obtain the following commutative diagram
		\[
			\xymatrix{
				&C_1\ar@{=}[r]\ar[d]&C_1\ar[d]\\
				A\ar[r]\ar@{=}[d]&D\ar[r]\ar[d]&P_C\ar[d]\ar@{-->}[r]&\\
				A\ar[r]&B\ar[r]\ar@{-->}[d]&C\ar@{-->}[r]\ar@{-->}[d]&\\
				&&&
			}
		\]
whose rows and columns are $\mathbb{E}$-triangles in $\xi$.
Since $P_C$ is $\xi$-projective, the $\mathbb{E}$-triangle $A\>D\>P_C\dashrightarrow$
 splits, and hence $D\simeq A\oplus P_C$. So the direct sum of
$A_1\longrightarrow G_A\longrightarrow A\dashrightarrow$
with the split $\mathbb{E}$-triangle
$0\longrightarrow P_C\xrightarrow{1_{P_C}}P_C\dashrightarrow$
is isomorphic to an $\mathbb E$-triangle
$A_1\longrightarrow G_A\oplus P_C\longrightarrow D\dashrightarrow$
in $\xi$. By Lemma \ref{Thm3.2} and $\rm(ET4)^{op}$, we have the following commutative diagram
		\[
			\xymatrix{
				A_1\ar@{=}[d] \ar[r] & B_1 \ar[r]\ar[d] & C_1 \ar[d] \ar@{-->}[r] &  \\
				A_1\ar[r] & G_A\oplus P_C\ar[d] \ar[r] & D\ar[d] \ar@{-->}[r]& \\
				& B\ar@{=}[r]\ar@{-->}[d]&B\ar@{-->}[d] \\
				& & &}
		\]
where all  rows and columns are  $\mathbb{E}$-triangles in $\xi$.  As $C_1\in\mathcal{QGP}(\xi)$, the induction
hypothesis implies
$\Qgpd{B_1}\leq n-1$.
It follows that $\Qgpd{B}\leq n$,
since $G_A\oplus P_C\in\mathcal{QGP}(\xi)$.

Conversely, suppose that
$\Qgpd{B}\leq n$.
There exists an $\mathbb E$-triangle
$B_1\longrightarrow G_B\longrightarrow B\dashrightarrow$
in $\xi$ with $G_B\in\mathcal{QGP}(\xi)$ and
$\Qgpd{B_1}\leq n-1$.
By Lemma~\ref{Thm3.2} and $(\mathrm{ET4})^{\mathrm{op}}$,
there is the following commutative diagram
		\[
			\xymatrix{
				B_1\ar@{=}[d] \ar[r] & H \ar[r]\ar[d] & A \ar[d] \ar@{-->}[r] &  \\
				B_1\ar[r] & G_B\ar[d] \ar[r] & B\ar[d] \ar@{-->}[r]& \\
				& C\ar@{=}[r]\ar@{-->}[d]&C\ar@{-->}[d] \\
				& & &}
		\]
where all  rows and columns are  $\mathbb{E}$-triangles in $\xi$. Since $G_B,C\in\mathcal{QGP}(\xi)$, we have
$H\in\mathcal{QGP}(\xi)$ by Lemma~\ref{HH}(1).
Hence $\Qgpd{A}\leq n$.
This completes the induction.

The last assertion follows from the equivalence, since
the quasi-$\xi$-$\mc{G}$projective dimension of a nonzero object
is either a nonnegative integer or infinite.
\end{proof}

\begin{cor}\label{cor:qgpd-projective-sum}
Let $M\in\Mcc$ and $P\in\P$. Then, for any integer $n\geq0$,
$\Qgpd{M}\leq n$ if and only if
$\Qgpd{M\oplus P}\leq n$.
In particular, if $M\neq0$, then
$\Qgpd{M\oplus P}=\Qgpd{M}$.
\end{cor}

\begin{proof}
Consider the split $\mathbb E$-triangle
$M\longrightarrow M\oplus P\longrightarrow P\dashrightarrow$
in $\xi$. Since $P\in\P\subseteq\QGP$, the assertion follows
from Lemma~\ref{lem4.4}.
\end{proof}

Let $\Mcx$ be a subcategory of $\Mcc$. We define
	\[
		\begin{aligned}
			& \mathcal{X}^{\perp}:=\left\{C \in \mathcal{C} \mid \ext^{n\geq 1}(X, C)=0 \text{ for all } X \in \mathcal{X}\right\} \\
			& ^{\perp} \mathcal{X}:=\left\{C \in \mathcal{C} \mid \ext^{n\geq 1}(C, X)=0 \text{ for all } X \in \mathcal{X} \right\}
		\end{aligned}
	\]
	For two subcategories $\Mcx$ and $\mathcal{Y}$ of $\Mcc$, we say $\Mcx\perp\mathcal{Y}$ if $\Mcx\subseteq{}^\perp\mathcal{Y}$ ( equivalently $\mathcal{Y}\subseteq \Mcx^\perp$ ).

	\begin{lem}\label{Lem4.6}
		$\P=\widehat{\mathcal{P}}(\xi)\bigcap{}^{\perp}\P$.
	\end{lem}
	\begin{proof}
		Obviously, $\P\subseteq\widehat{\mathcal{P}}(\xi)\bigcap{}^{\perp}\P$. Conversely, let $M\in \widehat{\mathcal{P}}(\xi)\bigcap{}^{\perp}\P$, and assume that $\pd{M}\leq n$. If $n=0$, it is obvious.  Hence we can assume $n\geq1$. So there exists a $\xi$-exact complex
		\[
			0\>P_n\>P_{n-1}\>\cdots\>P_1\>P_0\>M\>0
		\]
		with all $P_i\in\P$ and the $\xi$-resolution $\Mbe$-triangles $K_{i+1}\>P_i\>K_i\dashrightarrow$ (set $K_0=M$ and $K_n=P_n$). For any $P\in\P$, by Lemma \ref{LZHL}, there is an exact sequence
		\[
			\cdots\>\ext^{k}(M,P)\>\ext^{k}(P_0,P)\>\ext^{k}(K_1,P)\>\ext^{k+1}(M,P)\>\cdots
		\]
		Note that $\ext^{k}(M,P)=\ext^{k+1}(M,P)=0$ for $k\geq 1$, since $M\in {^{\bot}}\mathcal{P}(\xi)$ and $\ext^{k}(P_0,P)=0$. It follows that $K_1\in{}^{\perp}\P$. Repeating this process, we can obtain $K_i\in{}^{\perp}\P$ for all $1\leq i\leq n$. Consider the following $\Mbe$-triangle
		\[
			P_n\>P_{n-1}\>K_{n-1}\dashrightarrow
		\]
		in $\xi$. By Lemma \ref{LZHL} again, there is an exact sequence
		\[
			\Mcc(P_{n-1},P_n)\> \Mcc(P_n,P_n)\> \ext^1(K_{n-1},P_n)=0
		\]
       Consequently, the $\Mbe$-triangle $P_n\>P_{n-1}\>K_{n-1}\dashrightarrow$ is split. It follows that $ P_n\oplus K_{n-1}\cong P_{n-1}$ and hence $K_{n-1}\in\P$. Repeating this process, we can obtain $K_i\in\P$ for any $i\leq n$. Hence $M\in\P$ and $\widehat{\mathcal{P}}(\xi)\bigcap{}^{\perp}\P\subseteq \P$. Thus, $\P=\widehat{\mathcal{P}}(\xi)\bigcap{}^{\perp}\P$.
	\end{proof}
	\begin{thm}\label{thm4.4}
		Let $M$ be an object in $\widehat{\mathcal{QGP}}(\xi)$. Then the following statements are equivalent for any $n\geq 0$.

		(1) $\Qgpd{M}\leq n$.

		(2) For any $\xi$-exact complex
\[
    0\>K_n\>G_{n-1}\>\cdots\>G_0\>M\>0,
\]
if all $G_i\in\QGP$, then so is $K_n$.

		(3) For each $0\leq i\leq n$, there exists a $\xi$-exact complex
		\[
			0\>G_n\>G_{n-1}\>\cdots\> G_1\> G_0\>M\>0
		\]
		such that $G_i\in\QGP$ and $G_j\in\P$ for all $j\neq i$.

		(4) $\ext^{n+i}(M,P)=0$, $\forall P\in\P$, $\forall i\geq 1$.

		(5) $\ext^{n+i}(M,L)=0$, $\forall L\in\widehat{\mathcal{P}}(\xi)$, $\forall i\geq 1$.

		(6) $\ext^{n+i}(M,Q)=0$, $\forall Q\in\QP$, $\forall i\geq 1$.

		(7) $\ext^{n+i}(M,N)=0$, $\forall N\in\widehat{\mathcal{QP}}(\xi)$, $\forall i\geq 1$.\\
For $n=0$, conditions (2) and (3) reduce to $M\in\QGP$. Moreover, if $M\neq0$, then
\[
\begin{aligned}
\Qgpd{M}
&=\sup\Bigl(
\{0\}\cup \left\{m\in\mathbb{N}^+ \mid \exists P\in\P \text{\, such that \,} \ext^m(M,P)\neq 0 \right\}\Bigr)\\
				&=\sup\Bigl(\{0\}\cup \left\{m\in\mathbb{N}^+ \mid \exists L\in\widehat{\mathcal{P}}(\xi) \text{\, such that \,} \ext^m(M,L)\neq 0 \right\}\Bigr)\\
				&=\sup\Bigl(\{0\}\cup \left\{m\in\mathbb{N}^+ \mid \exists Q\in\QP \text{\, such that \,} \ext^m(M,Q)\neq 0 \right\}\Bigr)\\
				&=\sup\Bigl(\{0\}\cup \left\{m\in\mathbb{N}^+ \mid \exists N\in\widehat{\mathcal{QP}}(\xi) \text{\, such that \,} \ext^m(M,N)\neq 0 \right\}\Bigr).
\end{aligned}
\]
	\end{thm}

	\begin{proof}

For $n=0$, $(1)\Rightarrow(7)$ follows from
Proposition~\ref{TTTT}(1), while
$(7)\Rightarrow(5)\Rightarrow(4)$ and
$(7)\Rightarrow(6)\Rightarrow(4)$ are immediate.
The argument for $(4)\Rightarrow(1)$ below also covers $n=0$.
Thus we may assume $n\geq1$ in proving
$(1)\Rightarrow(2)$, $(2)\Rightarrow(1)$, and
$(1)\Rightarrow(7)$.

	(1) $\Rightarrow$ (2) By assumption, there exists an $\Mbe$-triangle
	$K\stackrel{x}\>G\stackrel{y}\>M\stackrel{\delta}\dashrightarrow$
	in $\xi$ with $G\in\QGP$ and $\Qgpd{K}\leq n-1$.
	Since $\mathcal{C}$ has enough $\xi$-projectives, there exists an
	$\Mbe$-triangle
	$L\stackrel{x^{\prime}}\>P\stackrel{y^{\prime}}\>M
	\stackrel{\delta^{\prime}}\dashrightarrow$
	in $\xi$ with $P\in\P$. Then there is a morphism $g:P\to G$
	such that $y^{\prime}=yg$. It follows from
	\cite[Lemma 3.6(1)]{JDP} and Lemma~\ref{Thm3.2} that there is
	a morphism $f$ such that the following diagram commutes
	\[
		\xymatrix{
			L \ar@{-->}[d]^{f} \ar[r]^{x'} &
			P \ar[d]^{g} \ar[r]^{y'} &
			M\ar@{=}[d]\ar@{-->}[r]^{\delta'}& \\
			K\ar[r]^{x} &
			G\ar[r]^{y} &
			M\ar@{-->}[r]^{\delta}&}
	\]
	and $L\>K\oplus P\>G\dashrightarrow$ is an
	$\mathbb E$-triangle in $\xi$. Note that $G\in\QGP$, so we have 
$\Qgpd{L}\leq n-1$ by Lemma~\ref{lem4.4}
and Corollary~\ref{cor:qgpd-projective-sum}.

Consider the $\Mbe$-triangle
$K_1\>G_0\>M\dashrightarrow$ in $\xi$. By the same argument,
there exists an $\mathbb E$-triangle
$L\>K_1\oplus P\>G_0\dashrightarrow$
in $\xi$. Hence Lemma~\ref{lem4.4} and
Corollary~\ref{cor:qgpd-projective-sum} yield
$\Qgpd{K_1}\leq n-1$.
By induction, we have $\Qgpd{K_n}\leq0$, and hence
$K_n\in\QGP$.

		(2) $\Rightarrow$ (1) Since $\Mcc$ has enough $\xi$-projectives, there exists a $\xi$-exact complex
		\[
			0\>K_n\>P_{n-1}\>\cdots \> P_1\> P_0\>M\>0
		\]
		with all $P_i\in\P$. Hence $K_n$ lies in $\QGP$ by hypothesis, which implies that $\Qgpd{M}\leq n$.

		(1) $\Rightarrow$ (7) Since $\Qgpd{M}\leq n$, there exists a $\xi$-exact complex
		\[
			0\>G_n\>G_{n-1}\>\cdots\>G_1\>G_0\>M\>0
		\]
		with all $G_i\in\QGP$ and the $\xi$-resolution $\mathbb{E}$-triangle $K_{i+1}\>G_i\>K_i\dashrightarrow$ (set $K_0=M$ and $K_n=G_n$) for any  $0\leq i\leq n-1$. For any $N\in\widehat{\mathcal{QP}}(\xi)$ and $k\geq 1$, by Lemma \ref{LZHL}, there is an exact sequence
		\[
			\cdots\>\ext^{k}(G_i,N)\>\ext^{k}(K_{i+1},N)\>\ext^{k+1}(K_i,N)\>\ext^{k+1}(G_i,N)\>\cdots
		\]
		Note that $\ext^{k\geq 1}(G_i,N)=0$ by Proposition \ref{TTTT}(1). Thus $\ext^{k+1}(K_i,N)\cong\ext^{k}(K_{i+1},N)$ for any $k\geq 1$. It follows that $\ext^{n+i}(M,N)\cong \ext^{i}(G_n,N)=0$ for any $i\geq 1$.

		(7) $\Rightarrow$ (5) $\Rightarrow$ (4) and (7) $\Rightarrow$ (6) $\Rightarrow$ (4) are obvious.

		(4) $\Rightarrow$ (1)  Since $M$ lies in $\widehat{\mathcal{QGP}}(\xi)$, it follows from Proposition \ref{Prop4.4} that there exists an $\Mbe$-triangle $K\>G\>M\dashrightarrow$ in $\xi$ with $G\in\QGP$ and $\pd{K}<\infty$. For any $P\in\P$ and $k\geq 1$, by Lemma \ref{LZHL}, there is an exact sequence
		\[
			\cdots\>\ext^{k}(G,P)\>\ext^{k}(K,P)\>\ext^{k+1}(M,P)\>\ext^{k+1}(G,P)\>\cdots
		\]
		Note that $\ext^{k\geq 1}(G,P)=0$ by Proposition \ref{TTTT}(1). So $\ext^{n+i}(K,P)\cong\ext^{n+i+1}(M,P)=0$ for any $i\geq 0 $ with $n+i\geq 1$.  
		
		If $n=0$, then we have $K\in\P$ by Lemma \ref{Lem4.6}. Since $K,G\in\QGP$ and $\ext^{1}(M,P)=0$ for any $P\in\P$, $M$ is in $\QGP$ by  Proposition \ref{Prop3.7}. This proves the case $n=0$.

		Now assume $n\geq 1$ and let  $\pd{K}=m$. We only need to consider the case $m\geq n$. Then there exists a $\xi$-exact complex
		\[
			0\>P_m\>P_{m-1}\>\cdots\>P_1\>P_0\>K\>0
		\]
		with all $P_i\in\P$ and the $\xi$-resolution $\Mbe$-triangles $K_{i+1}\>P_i\>K_i\dashrightarrow$ (set $K_0=K$ and $K_m=P_m$). By Lemma \ref{LZHL}, for all $k\geq 1$ and any $P\in \P$, there is an exact sequence
		\[
			\cdots\> \ext^k(P_i,P)\>\ext^k(K_{i+1},P)\>\ext^{k+1}(K_i,P)\> \ext^{k+1}(P_i,P) \>\cdots
		\]
		Since $\ext^k(P_i,P)=\ext^{k+1}(P_i,P)=0$, we have $\ext^k(K_{i+1},P)\cong\ext^{k+1}(K_i,P)$. Thus, we obtain  $\ext^i(K_{n-1},P)\cong\ext^{i+n-1}(K,P)=0$ for any $i\geq 1$, which means $K_{n-1}\in{}^{\perp} \P$. Note that $K_{n-1}\in\widehat{\mathcal{P}}(\xi)$, so $K_{n-1}$ is in $\P$ by Lemma \ref{Lem4.6}. Therefore, $\Qgpd{K}\leq n-1$ and $\Qgpd{M}\leq n$.

		(1) $\Leftrightarrow$ (3) $\Leftrightarrow$ (6) It follows from \cite[Theorem 5.3]{He}.

        For the last assertion, put $d=\Qgpd{M}$.
By $(1)\Rightarrow(7)$,
$\ext^{m}(M,N)=0$ for every
$N\in\widehat{\mathcal{QP}}(\xi)$ and every $m>d$.
If $d=0$, the assertion follows immediately.

Let $d>0$. If $\ext^{d}(M,P)=0$ for every $P\in\P$,
then condition~(4) holds with $n=d-1$. Hence
$\Qgpd{M}\leq d-1$, a contradiction. Therefore,
$\ext^{d}(M,P)\neq0$ for some $P\in\P$.
Since $\P$ is contained in
$\widehat{\mathcal P}(\xi)$, $\QP$, and
$\widehat{\mathcal{QP}}(\xi)$, the four suprema are equal
to $d$.
	\end{proof}

Since $\P\subseteq\QGP\subseteq\GP$, one has
$\xi\text{-}\mc{G}\text{pd}(M)\leq\Qgpd{M}\leq\pd{M}$
for every $M\in\Mcc$. If
$M\in\widehat{\mathcal P}(\xi)$, then
$\xi\text{-}\mc{G}\text{pd}(M)=\pd{M}$ by
\cite[Proposition~5.4]{JDP}. Hence
$\Qgpd{M}=\pd{M}$.
The following result is a direct consequence of
Theorem~\ref{thm4.4}.	
	
	\begin{cor}
		If $0\neq M$ has finite $\xi$-projective dimension, then
		\[
			\begin{split}
				\pd{M}&=\sup\Bigl(
\{0\}\cup \left\{m\in\mathbb{N}^+ \mid \exists P\in\P \text{\, such that \,} \ext^m(M,P)\neq 0 \right\}\Bigr)\\
				&=\sup\Bigl(\{0\}\cup \left\{m\in\mathbb{N}^+ \mid \exists L\in\widehat{\mathcal{P}}(\xi) \text{\, such that \,} \ext^m(M,L)\neq 0 \right\}\Bigr)\\
				&=\sup\Bigl(\{0\}\cup \left\{m\in\mathbb{N}^+ \mid \exists Q\in\QP \text{\, such that \,} \ext^m(M,Q)\neq 0 \right\}\Bigr)\\
				&=\sup\Bigl(\{0\}\cup \left\{m\in\mathbb{N}^+ \mid \exists N\in\widehat{\mathcal{QP}}(\xi) \text{\, such that \,} \ext^m(M,N)\neq 0 \right\}\Bigr).
			\end{split}
		\]
	\end{cor}

	\begin{cor}\label{LYF}
		The following equalities hold.

		(1) $\P=\QGP \bigcap \widehat{\mathcal{P}}(\xi)$.

		(2) $\QGP=\widehat{\mc{QGP}}(\xi)\bigcap{} ^{\perp}\widehat{\mathcal{P}}(\xi)=\widehat{\mc{QGP}}(\xi)\bigcap{}^{\perp}\P$.

		(3) $\QGP=\widehat{\mc{QGP}}(\xi)\bigcap{}^{\perp}\widehat{\mathcal{QP}}(\xi)=\widehat{\mc{QGP}}(\xi)\bigcap{}^{\perp}\QP$.
	\end{cor}

	Dually, we have the following results.

	\begin{thm}\label{thm4.8}
		Let $M$ be an object in $\widehat{\mathcal{QGI}}(\xi)$. Then the following statements are equivalent for any  $n\geq 0$.

		(1) $\Qgid{M}\leq n$.

		(2) For any $\xi$-exact complex
		\[
			0\>M\>G_0\>G_{-1}\>\cdots\>G_{-(n-1)}\>K_{-n}\>0,
		\]
		if all $G_i\in\QGI$, then so is $K_{-n}$.

		(3) For each $0\leq i\leq n$, there exists a $\xi$-exact complex
		\[
			0\>M\>G_{0}\>G_{-1}\>\cdots\> G_{-(n-1)}\>G_{-n}\>0
		\]
		such that $G_{-i}\in\QGI$ and $G_{-j}\in\mathcal{I}(\xi)$ for all $j\neq i$.

		(4) $\ext^{n+i}(I,M)=0$, $\forall I\in\mathcal{I}(\xi)$, $\forall i\geq 1$.

		(5) $\ext^{n+i}(L,M)=0$, $\forall L\in\widehat{\mathcal{I}}(\xi)$, $\forall i\geq 1$.

		(6) $\ext^{n+i}(Q,M)=0$, $\forall Q\in\QI$, $\forall i\geq 1$. 

		(7) $\ext^{n+i}(N,M)=0$, $\forall N\in\widehat{\mathcal{QI}}(\xi)$, $\forall i\geq 1$.\\
For $n=0$, conditions (2) and (3) reduce to $M\in\QGI$.

		Moreover, for any $M\neq 0$, we have
		\[
			\begin{split}
				\Qgid{M}&=\sup\Bigl(
\{0\}\cup \left\{m\in\mathbb{N}^+ \mid \exists I\in\mathcal{I}(\xi) \text{\, such that \,} \ext^m(I,M)\neq 0 \right\}\Bigr)\\
				&=\sup\Bigl(
\{0\}\cup \left\{m\in\mathbb{N}^+ \mid \exists L\in\widehat{\mathcal{I}}(\xi) \text{\, such that \,} \ext^m(L,M)\neq 0 \right\}\Bigr)\\
				&=\sup\Bigl(
\{0\}\cup \left\{m\in\mathbb{N}^+ \mid \exists Q\in\QI \text{\, such that \,} \ext^m(Q,M)\neq 0 \right\}\Bigr)\\
				&=\sup\Bigl(
\{0\}\cup \left\{m\in\mathbb{N}^+ \mid \exists N\in\widehat{\mathcal{QI}}(\xi) \text{\, such that \,} \ext^m(N,M)\neq 0 \right\}\Bigr).
			\end{split}
		\]
	\end{thm}

	\begin{cor}
	If $0\neq M$ has finite $\xi$-injective dimension, then
		\[
			\begin{split}
				\id{M}&=\sup\Bigl(
\{0\}\cup \left\{m\in\mathbb{N}^+ \mid \exists I\in\mathcal{I}(\xi) \text{\, such that \,} \ext^m(I,M)\neq 0 \right\}\Bigr)\\
				&=\sup\Bigl(
\{0\}\cup \left\{m\in\mathbb{N}^+ \mid \exists L\in\widehat{\mathcal{I}}(\xi) \text{\, such that \,} \ext^m(L,M)\neq 0 \right\}\Bigr)\\
				&=\sup\Bigl(
\{0\}\cup \left\{m\in\mathbb{N}^+ \mid \exists Q\in\QI \text{\, such that \,} \ext^m(Q,M)\neq 0 \right\}\Bigr)\\
				&=\sup\Bigl(
\{0\}\cup \left\{m\in\mathbb{N}^+ \mid \exists N\in\widehat{\mathcal{QI}}(\xi) \text{\, such that \,} \ext^m(N,M)\neq 0 \right\}\Bigr).
			\end{split}
		\]
	\end{cor}

	\begin{cor}
		The following equalities hold.

		(1) $\mc{I}(\xi)=\QGI \bigcap \widehat{\mathcal{I}}(\xi)$.

		(2) $\QGI=\widehat{\mc{QGI}}(\xi)\bigcap\widehat{\mathcal{I}}(\xi)^{\perp}=\widehat{\mc{QGI}}(\xi)\bigcap\mc{I}(\xi)^{\perp}$.

		(3) $\QGI=\widehat{\mc{QGI}}(\xi)\bigcap{}\widehat{\mathcal{QI}}(\xi)^{\perp}=\widehat{\mc{QGI}}(\xi)\bigcap\QI^{\perp}$.
	\end{cor}
\makeatletter
\@ifundefined{problem}{%
  \@ifundefined{c@defn}{%
    \newtheorem{problem}{Problem}[section]%
  }{%
    \newtheorem{problem}[defn]{Problem}%
  }%
}{}
\makeatother

\section{Applications to module categories}\label{sec:applications}
In this section, we study quasi-Gorenstein projective
modules and their homological dimensions.
Let $R$ be an associative ring with identity, and let
$\mathcal C=R\text{-}\operatorname{Mod}$ be the category
of left $R$-modules. We equip $\mathcal C$ with its
canonical exact structure and take $\xi$ to be the
class of all short exact sequences. The relative
extension groups are then the ordinary groups
$\operatorname{Ext}_R^i(-,-)$.

We denote by $\operatorname{Proj}R$, $\mathcal{QP}(R)$
and $\mathcal{QGP}(R)$ the full subcategories of
projective, quasi-projective and quasi-Gorenstein
projective $R$-modules, respectively. The corresponding
dimensions are denoted by $\operatorname{pd}_R$,
$\mathcal Q\text{pd}_R$ and $\mathcal{QG}\text{pd}_R$.
As before, $\widehat{\mathcal{QP}}(R)$ and
$\widehat{\mathcal{QGP}}(R)$ denote the full subcategories
of modules of finite quasi-projective and
quasi-Gorenstein projective dimension, respectively.
For an $R$-module $T$, let $\operatorname{Add}_R T$
and $\operatorname{add}_R T$ denote the full
subcategories consisting of direct summands of
arbitrary and finite direct sums of copies of $T$,
respectively.

We first apply the preceding results to compare
quasi-Gorenstein projective dimension with
projective dimension.

\begin{prop}\label{comparison}
Let $M$ be a left $R$-module such that
$\operatorname{Ext}_R^1(Q,M)=0$ for every
$Q\in\mathcal{QGP}(R)$. Then
\[
    \QGpdR{M}=\pdR{M}.
\]
Moreover,
\[
    \widehat{\mathcal{QGP}}(R)
    \cap\widehat{\mathcal{QP}}(R)
    =
    \bigl\{N\in R\text{-}\operatorname{Mod}
    \bigm|\pdR{N}<\infty\bigr\}.
\]
In particular, every nonprojective
$G\in\mathcal{QGP}(R)$ satisfies
$\QpdR{G}=\infty$.
\end{prop}
\begin{proof}
    Both dimensions are $-1$ when $M=0$. Since
    $\operatorname{Proj}R\subseteq\mathcal{QGP}(R)$,
    \[
        \operatorname{QGpd}_RM\leq\operatorname{pd}_RM.
    \]
    If $\QGpdR{M}=\infty$, then
    $\pdR{M}=\infty$. Suppose that
    $\QGpdR{M}=n<\infty$. By
    Lemma~\ref{GFC}, there is a short exact sequence
    \[
        0\longrightarrow M\longrightarrow L
        \longrightarrow Q\longrightarrow0
    \]
    with $Q\in\mathcal{QGP}(R)$ and
    $\pdR{L}\leq n$. The sequence splits, so
    $\pdR{M}\leq n$, and the first assertion
    follows.

    If
    $N\in\widehat{\mathcal{QGP}}(R)
       \cap\widehat{\mathcal{QP}}(R)$, then 
    $\operatorname{Ext}_R^1(Q,N)=0$ for every
    $Q\in\mathcal{QGP}(R)$ by  Proposition~\ref{TTTT}\textup{(1)}. Hence
    $\pdR{N}=\QGpdR{N}<\infty$.
    The reverse inclusion follows from
    $\operatorname{Proj}R\subseteq
      \mathcal{QGP}(R)\cap\mathcal{QP}(R)$.
\end{proof}
\begin{thm}\label{thm:perfect}
If $R$ is left perfect, then $\mathcal{QGP}(R)=\operatorname{Proj}R$.
\end{thm}
\begin{proof}
Let $G\in\mathcal{QGP}(R)$. Every simple left
$R$-module is quasi-projective. Hence  $\operatorname{Ext}_R^1(G,S)=0$ for every simple left $R$-module $S$ by Proposition~\ref{TTTT}(1).

Let $J$ denote the Jacobson radical of $R$. By
Bass's characterization~\cite[Theorem~P]{Bass}, every
left $R$-module has a projective cover, $R/J$ is
semisimple, and for every sequence $a_1,a_2,\ldots$ in
$J$ there exists $n\geq1$ such that $a_1a_2\cdots a_n=0$.

Let
\[
    0\longrightarrow K\xrightarrow{i}P
    \longrightarrow G\longrightarrow0
\]
be a projective cover of $G$. Then $K$ is superfluous
in $P$. Suppose that $K\neq0$. If $K/JK=0$, then $JK=K$.
Fix $0\neq k_0\in K$. Since $k_0\in JK$, there exist
$a_{1j}\in J$ and $k_{1j}\in K$ such that
\[
    k_0=\sum_{j=1}^{m_1}a_{1j}k_{1j}.
\]
Thus $a_{1j}k_{1j}\neq0$ for some $j$. Set
$a_1=a_{1j}$ and $k_1=k_{1j}$. Inductively, suppose
that 
\[a_1\cdots a_{n-1}k_{n-1}\neq0.\]
Since $k_{n-1}\in JK$, there exist $a_{nj}\in J$ and
$k_{nj}\in K$ such that
\[
    k_{n-1}=\sum_{j=1}^{m_n}a_{nj}k_{nj}.
\]
Hence $a_1\cdots a_{n-1}a_{nj}k_{nj}\neq0$
for some $j$. Set $a_n=a_{nj}$ and $k_n=k_{nj}$.
Proceeding inductively, we obtain a sequence
$a_1,a_2,\ldots$ in $J$ satisfying
\[
    a_1a_2\cdots a_nk_n\neq0,\quad
    n\geq1.
\]
In particular, $a_1a_2\cdots a_n\neq0$ for every
$n\geq1$, contrary to the preceding property of $J$.
Thus we obtain $K/JK\neq0$.

Since $K/JK\neq0$ and $R/J$ is semisimple, $K/JK$
has a nonzero simple quotient $S$. Thus there is an
epimorphism $\overline f:K/JK\longrightarrow S$.
Via the canonical epimorphism $R\to R/J$, $S$ is also
a nonzero simple left $R$-module. If
$q:K\to K/JK$ denotes the canonical epimorphism, set
\[
    f=\overline f q:K\longrightarrow S.
\]
Then $f$ is an epimorphism. Note that $\operatorname{Ext}_R^1(G,S)=0$. There is the following short exact sequence
\[
   0\> \operatorname{Hom}_R(G,S)\>\operatorname{Hom}_R(P,S)
    \longrightarrow
    \operatorname{Hom}_R(K,S)\>0.
\] 
Hence there exists $h:P\to S$ such that $f=hi$.
Since $f$ is surjective, so is $h$. Therefore $P/\operatorname{Ker}h\cong S$,
and $\operatorname{Ker}h$ is a maximal submodule of
$P$.

If $K\nsubseteq\operatorname{Ker}h$, then the
maximality of $\operatorname{Ker}h$ implies $K+\operatorname{Ker}h=P$.
Note that  $K$ is superfluous in $P$. So we have
$\operatorname{Ker}h=P$ and $S\cong P/\operatorname{Ker}h=0$, a contradiction. Hence $K\subseteq\operatorname{Ker}h$.
It follows that $f=hi=0$, contradicting the
surjectivity of $f$. Thus $K=0$, and hence $G\cong P$
is projective. Therefore $\mathcal{QGP}(R)=\operatorname{Proj}R$.
\end{proof}

Throughout the rest of this section, $R$ is commutative unless
otherwise stated, and modules over quotient rings of $R$ are
viewed as $R$-modules via the canonical quotient maps.

For a complete quasi-projective resolution
$\mathbf P=(P_i,d_i)_{i\in\mathbb Z}$, define the coefficient class
\[
    \mathcal E(\mathbf P)
    =
    \bigl\{
        N\in R\text{-}\operatorname{Mod}
        \bigm|
        \operatorname{Hom}_R(\mathbf P,N)
        \text{ is exact}
    \bigr\}.
\]
For each $i\in\mathbb Z$, put $Z_i=\operatorname{Im}d_i$.
Then $Z_i\in\mathcal{QGP}(R)$, and there is a short
exact sequence
\[
    0\longrightarrow Z_{i+1}\longrightarrow P_i
    \longrightarrow Z_i\longrightarrow0.
\]
Let $N\in\widehat{\mathcal{QP}}(R)$.
By Proposition~\ref{TTTT}(1), we have
$\operatorname{Ext}_R^1(Z_i,N)=0$ for every $i\in\mathbb Z$.
Applying $\operatorname{Hom}_R(-,N)$ to the above short
exact sequence, we obtain the following exact sequence:
\[
    0\longrightarrow\operatorname{Hom}_R(Z_i,N)
    \longrightarrow\operatorname{Hom}_R(P_i,N)
    \longrightarrow\operatorname{Hom}_R(Z_{i+1},N)
    \longrightarrow0.
\]
Pasting these exact sequences together, we obtain that
$\operatorname{Hom}_R(\mathbf P,N)$ is exact. Hence $\widehat{\mathcal{QP}}(R)\subseteq  \mathcal E(\mathbf P)$.

\begin{lem}\label{lem:coeff}
 Let $\mathbf P$ be a complete quasi-projective resolution. Then
$\mathcal E(\mathbf P)$ is closed under direct summands. For every
short exact sequence
\[
    0\longrightarrow N'\longrightarrow N\longrightarrow N''\longrightarrow0,
\]
if two of $N'$, $N$, and $N''$ belong to
$\mathcal E(\mathbf P)$, then so does the third. Moreover, $\operatorname{Add}_R T\subseteq\mathcal E(\mathbf P)$
for every finitely generated $R$-module $T$.
\end{lem}

\begin{proof}
    Let $N\in\mathcal E(\mathbf P)$ and suppose that
    $N=N'\oplus N''$. Then there is an isomorphism
    of complexes
    \[
        \operatorname{Hom}_R(\mathbf P,N)
        \cong
        \operatorname{Hom}_R(\mathbf P,N')
        \oplus
        \operatorname{Hom}_R(\mathbf P,N'').
    \]
    Since $\operatorname{Hom}_R(\mathbf P,N)$ is exact,
    both complexes on the right are exact. Thus
    $N',N''\in\mathcal E(\mathbf P)$, and
    $\mathcal E(\mathbf P)$ is closed under direct summands.

    Now let
    \[
        0\longrightarrow N'\longrightarrow N
        \longrightarrow N''\longrightarrow0
    \]
    be a short exact sequence. Since each $P_i$ is
    projective, applying $\operatorname{Hom}_R(P_i,-)$
    gives a short exact sequence
    \[
        0\longrightarrow\operatorname{Hom}_R(P_i,N')
        \longrightarrow\operatorname{Hom}_R(P_i,N)
        \longrightarrow\operatorname{Hom}_R(P_i,N'')
        \longrightarrow0
    \]
    for every $i\in\mathbb Z$. Hence there is a short
    exact sequence of complexes
    \[
        0\longrightarrow\operatorname{Hom}_R(\mathbf P,N')
        \longrightarrow\operatorname{Hom}_R(\mathbf P,N)
        \longrightarrow\operatorname{Hom}_R(\mathbf P,N'')
        \longrightarrow0.
    \]
    By the induced long exact sequence in cohomology,
    if two of these complexes are exact, then so is
    the third. In particular, $\mathcal E(\mathbf P)$ is closed
    under extensions.

    Let $T$ be generated by $t_1,\ldots,t_s$. Set $T_0=0$
and $T_j=Rt_1+\cdots+Rt_j$ for $1\leq j\leq s$.
Then $T_j/T_{j-1}\cong R/I_j$, where
$I_j=\{r\in R\mid rt_j\in T_{j-1}\}$.
For any set $\Lambda$, taking direct sums gives
short exact sequences
\[
    0\longrightarrow T_{j-1}^{(\Lambda)}
    \longrightarrow T_j^{(\Lambda)}
    \longrightarrow (R/I_j)^{(\Lambda)}
    \longrightarrow0,
    \quad 1\leq j\leq s.
\]
For each $j$, let $U$ be an $R$-submodule of
$(R/I_j)^{(\Lambda)}$, and consider the canonical epimorphism
\[
    \pi:(R/I_j)^{(\Lambda)}
    \twoheadrightarrow (R/I_j)^{(\Lambda)}/U.
\]
Since $I_j$ annihilates $(R/I_j)^{(\Lambda)}$, both $U$
and the quotient $(R/I_j)^{(\Lambda)}/U$ are $R/I_j$-modules. Moreover, $\pi$
and any $R$-linear map $\varphi: (R/I_j)^{(\Lambda)}
    \rightarrow (R/I_j)^{(\Lambda)}/U$ are $R/I_j$-linear.
Since $(R/I_j)^{(\Lambda)}$ is free over $R/I_j$,
there exists an $R/I_j$-linear endomorphism $h$ such
that $\pi h=\varphi$. As $h$ is also $R$-linear,
$(R/I_j)^{(\Lambda)}$ is quasi-projective over $R$.
Therefore,
\[
    (R/I_j)^{(\Lambda)}
    \in\mathcal{QP}(R)\subseteq\mathcal E(\mathbf P).
\]
    Since $T_0^{(\Lambda)}=0$ and
    $\mathcal E(\mathbf P)$ is closed under extensions,
    induction on $j$ gives
    $T_j^{(\Lambda)}\in \mathcal E(\mathbf P)$ for all
    $0\leq j\leq s$. In particular,
    $T^{(\Lambda)}\in\mathcal E(\mathbf P)$.
    The closure under direct summands now implies $\operatorname{Add}_RT\subseteq\mathcal E(\mathbf P)$.
\end{proof}

\begin{cor}\label{cor:quotient-coefficients}
Let $\mathbf P$ be a complete quasi-projective resolution over $R$.
Then the following statements hold.

(1) For every ideal $I$ of $R$, each
$R/I$-module $N$ with $\operatorname{pd}_{R/I}N<\infty$
belongs to $\mathcal E(\mathbf P)$.

(2) If $S$ is an Artinian quotient of
$R$, then every $S$-module belongs to $\mathcal E(\mathbf P)$.
\end{cor}

\begin{proof}
   (1) Let $I$ be an ideal of $R$, and let $N$ be an
$R/I$-module with $\operatorname{pd}_{R/I}N<\infty$.
Since $R/I$ is finitely generated as an $R$-module, it
follows from Lemma~\ref{lem:coeff} that $\operatorname{Add}_R(R/I)\subseteq\mathcal E(\mathbf P)$.
Let $Q$ be a projective $R/I$-module. Then $Q$ is a direct
summand of $(R/I)^{(\Lambda)}$ for some set $\Lambda$, and
hence
\[
    Q\in\operatorname{Add}_R(R/I)
    \subseteq\mathcal E(\mathbf P).
\]
Since $\operatorname{pd}_{R/I}N<\infty$, there exists a finite
projective resolution
\[
    0\longrightarrow Q_n\longrightarrow Q_{n-1}
    \longrightarrow\cdots\longrightarrow Q_0
    \longrightarrow N\longrightarrow0
\]
over $R/I$.  If $n=0$, then $N\cong Q_0$ and hence
$N\in\mathcal E(\mathbf P)$. Now assume that $n\geq1$.
Let $L_0=N$ and
$L_{i+1}=\operatorname{Ker}(Q_i\longrightarrow L_i)$
for $0\leq i<n$. Then $L_n\cong Q_n$, and for each
$0\leq i<n$ there is a short exact sequence
\[
    0\longrightarrow L_{i+1}\longrightarrow Q_i
    \longrightarrow L_i\longrightarrow0.
\]
Since $L_n,Q_i\in\mathcal E(\mathbf P)$ for $0\leq i<n$,
it follows from Lemma~\ref{lem:coeff} by descending
induction that $L_i\in\mathcal E(\mathbf P)$ for
$0\leq i\leq n$. Hence $N=L_0\in\mathcal E(\mathbf P)$.

(2)    Let $S=R/I$ be an Artinian quotient of $R$, and let $J$ be the Jacobson radical of $S$. Since $S$ is Artinian, $J$ is nilpotent,
    and hence $J^h=0$ for some $h\geq1$.
    Moreover, $S/J$ is semisimple. Thus, for every
    $S$-module $N$, there is a finite filtration
    \[
        0=J^hN\subseteq J^{h-1}N\subseteq\cdots
        \subseteq JN\subseteq N.
    \]
For each $0\leq i<h$, the quotient module
    $J^iN/J^{i+1}N$ is annihilated by $J$, and hence is
    naturally an $S/J$-module. Since $S/J$ is semisimple,
    $J^iN/J^{i+1}N$ is semisimple as an $S/J$-module.
    Moreover, the canonical homomorphism  $R\longrightarrow S/J$ is surjective. Therefore, the $R$-submodules of
    $J^iN/J^{i+1}N$ coincide with its $S/J$-submodules.
    It follows that $J^iN/J^{i+1}N$ is also semisimple
    as an $R$-module. Hence, for every $R$-submodule
    $L\subseteq J^iN/J^{i+1}N$, the canonical projection
    \[
        J^iN/J^{i+1}N
        \longrightarrow
        \bigl(J^iN/J^{i+1}N\bigr)/L
    \]
    splits. Thus $J^iN/J^{i+1}N$ is quasi-projective
    as an $R$-module, and consequently
    \[
        J^iN/J^{i+1}N
        \in\mathcal{QP}(R)\subseteq\mathcal E(\mathbf P).
    \]
For each $0\leq i<h$, there is a short exact sequence
    \[
        0\longrightarrow J^{i+1}N
        \longrightarrow J^iN
        \longrightarrow J^iN/J^{i+1}N
        \longrightarrow0.
    \]
Since $J^hN=0\in\mathcal E(\mathbf P)$ and
$\mathcal E(\mathbf P)$ is closed under extensions by
Lemma~\ref{lem:coeff}, descending induction along the above
filtration yields $N\in\mathcal E(\mathbf P)$.
\end{proof}

\begin{cor}\label{cor:fg-qgp}
Every finitely generated module in $\mathcal{QGP}(R)$ is projective.
\end{cor}
\begin{proof}
Let $G\in\mathcal{QGP}(R)$ be finitely generated. By definition,
there is a complete quasi-projective resolution
$\mathbf P=(P_i,d_i)_{i\in\mathbb Z}$ such that
$G=\operatorname{Coker}d_1$. By Lemma~\ref{lem:coeff},
$G\in\mathcal E(\mathbf P)$. Hence the short exact sequence
\[
    0\longrightarrow G\longrightarrow P_{-1}
    \longrightarrow\operatorname{Coker}d_0\longrightarrow0
\]
is $\operatorname{Hom}_R(-,G)$-exact, and therefore splits.
Thus $G$ is projective.
\end{proof}

By the Raynaud--Gruson theorem
\cite[II, Corollary~3.2.7]{RG}
(see also \cite[Theorem~4.1]{NY}), every flat module $F$
over a commutative Noetherian ring $S$ of finite Krull
dimension satisfies $\operatorname{pd}_S F\leq \dim S$.

We denote by $\operatorname{Spec}R$ the set of prime ideals
of $R$. For $\mathfrak p\in\operatorname{Spec}R$, let
$R_{\mathfrak p}=(R\setminus\mathfrak p)^{-1}R$ be the
localization of $R$ at $\mathfrak p$. For an $R$-module $M$,
set $M_{\mathfrak p}=R_{\mathfrak p}\otimes_R M$. Then
$R_{\mathfrak p}$ is a local ring with maximal ideal
$\mathfrak pR_{\mathfrak p}$, and its residue field is
denoted by
$\kappa(\mathfrak p)
=R_{\mathfrak p}/\mathfrak pR_{\mathfrak p}
\cong\operatorname{Frac}(R/\mathfrak p)$.

By \cite[Corollary~5]{Hashimoto}, we have the following.

\begin{lem}\label{lem:hashimoto}
Let $R$ be a commutative Noetherian ring and let $\mathbf X$
be a complex of projective $R$-modules. If $\kappa(\mathfrak p)\otimes_R\mathbf X $ is exact for every $\mathfrak p\in\operatorname{Spec}R$, then
$\mathbf X$ is null-homotopic; equivalently, it is contractible.
\end{lem}

\begin{thm}\label{thm:finite-krull}
Let $R$ be a commutative Noetherian ring of finite Krull
dimension. Then
$\mathcal{QGP}(R)=\operatorname{Proj}R$. Moreover, every
complete quasi-projective resolution over $R$ is
contractible.
\end{thm}

\begin{proof}
Let $\mathbf P=(P_i,d_i)_{i\in\mathbb Z}$ be an arbitrary
complete quasi-projective resolution, and put
$G=\operatorname{Coker}d_1$. For each
$\mathfrak p\in\operatorname{Spec}R$, the residue field
$\kappa(\mathfrak p)$ is flat over $R/\mathfrak p$. Note that 
$R/\mathfrak p$ is a commutative Noetherian ring and
$\dim(R/\mathfrak p)\leq\dim R<\infty$. So we have
$\operatorname{pd}_{R/\mathfrak p}\kappa(\mathfrak p)
<\infty$ by the Raynaud--Gruson theorem. It follows from
Corollary~\ref{cor:quotient-coefficients}\textup{(1)} that
$\kappa(\mathfrak p)\in\mathcal E(\mathbf P)$, and hence
$\operatorname{Hom}_R(\mathbf P,\kappa(\mathfrak p))$ is
exact.

By the tensor--Hom adjunction, there is a natural
isomorphism of complexes
\[
    \operatorname{Hom}_R
    \bigl(\mathbf P,\kappa(\mathfrak p)\bigr)
    \cong
    \operatorname{Hom}_{\kappa(\mathfrak p)}
    \bigl(\kappa(\mathfrak p)\otimes_R\mathbf P,
          \kappa(\mathfrak p)\bigr).
\]
Thus the complex on the right is exact.  Note that linear duality over the field $\kappa(\mathfrak p)$ is exact, the cohomology of this Hom-complex is naturally isomorphic to the dual of the corresponding homology of $\kappa(\mathfrak p)\otimes_R P$. Since a vector space is zero whenever its dual is zero,
$\kappa(\mathfrak p)\otimes_R\mathbf P$ is exact. Therefore, $\mathbf P$ is
contractible by Lemma~\ref{lem:hashimoto}. Thus $\mathbf P$ is split exact. In particular,
$G\cong\operatorname{Im}d_0$ is a direct summand of
$P_{-1}$, and hence $G$ is projective. This proves
$\mathcal{QGP}(R)\subseteq\operatorname{Proj}R$. The
reverse inclusion follows from the definition.
\end{proof}

Let $S$ be a commutative ring. For an $S$-module $T$, the
ideal
$\operatorname{tr}_S(T)
=\sum_{f\in\operatorname{Hom}_S(T,S)}f(T)$
is called the trace ideal of $T$. If
$\Gamma=\operatorname{End}_S(T)$, then
$\operatorname{Hom}_S(T,N)$ is regarded as a right
$\Gamma$-module by precomposition for every $S$-module $N$.
All projective dimensions over endomorphism rings below
are taken for right modules.

\begin{lem}\label{lem:trace}
Let $S$ be a commutative ring, $T$ a finitely generated
$S$-module, and $N$ an $S$-module. Set
$\Gamma=\operatorname{End}_S(T)$. Suppose that
$\operatorname{pd}_{\Gamma}\operatorname{Hom}_S(T,N)\leq d$
for some integer $d\geq0$. Then there exists an augmented
complex
\[
    \mathbf A:\quad
    0\longrightarrow A_d\longrightarrow\cdots
    \longrightarrow A_0\longrightarrow N\longrightarrow0
\]
with $N$ in degree $-1$ and
$A_i\in\operatorname{Add}_S T$ for $0\leq i\leq d$, such
that $\operatorname{Hom}_S(T,\mathbf A)$ is exact and
$\operatorname{tr}_S(T)H_i(\mathbf A)=0$ for every
$i\in\mathbb Z$. Moreover, if $\lambda$ is an infinite cardinal with
$|S|,|N|\leq\lambda$, then $\mathbf A$ may be chosen
so that each $A_i$ is a direct summand of
$T^{(\lambda)}$ and $|H_i(\mathbf A)|\leq\lambda$
for all $i$.
\end{lem}

\begin{proof}
Let $B=\operatorname{Hom}_S(T,N)$, regarded as a right
$\Gamma$-module by precomposition, and regard $T$ as a
left $\Gamma$-module by evaluation. By hypothesis, there
is an exact sequence
\[
    0\longrightarrow Q_d\longrightarrow\cdots
    \longrightarrow Q_0\xrightarrow{\varepsilon}B
    \longrightarrow0
\]
of right $\Gamma$-modules with each $Q_i$ projective.
Set $A_i=Q_i\otimes_{\Gamma}T$ for $0\leq i\leq d$.
Since $Q_i$ is a direct summand of a free right
$\Gamma$-module, $A_i\in\operatorname{Add}_S T$.
Tensor the chosen resolution with $T$ and compose its
augmentation with the evaluation map
$B\otimes_{\Gamma}T\to N$, $f\otimes t\mapsto f(t)$.
This defines an augmented complex $\mathbf A$, whose
last differential $\partial_0:A_0\to N$ satisfies
$\partial_0(q\otimes t)=\varepsilon(q)(t)$.

Since $T$ is finitely generated, the functor
$\operatorname{Hom}_S(T,-)$ commutes with arbitrary
direct sums. Hence the canonical map
$F\longrightarrow\operatorname{Hom}_S
(T,F\otimes_{\Gamma}T)$ is an isomorphism for any
free right $\Gamma$-module $F$.
Each $Q_i$ is a direct summand of a free right
$\Gamma$-module, so the canonical maps
$\eta_i:Q_i\longrightarrow\operatorname{Hom}_S(T,A_i)$,
defined by $\eta_i(q)(t)=q\otimes t$, are isomorphisms.
These maps commute with the differentials by naturality.
Note that $\partial_0\circ\eta_0(q)=\varepsilon(q)$
for every $q\in Q_0$, so they also commute with the
augmentations. It follows that
$\operatorname{Hom}_S(T,\mathbf A)$ is isomorphic to
the chosen projective resolution of $B$ and is
therefore exact.

Put $A_{-1}=N$ and $A_{d+1}=A_{-2}=0$, and denote the
differentials of $\mathbf A$ by
$\partial_i:A_i\to A_{i-1}$. Let $-1\leq i\leq d$,
$z\in\operatorname{Ker}\partial_i$, and
$f\in\operatorname{Hom}_S(T,S)$. The homomorphism
$T\to A_i$, $t\mapsto f(t)z$, is a cycle in
$\operatorname{Hom}_S(T,\mathbf A)$. Since this complex
is exact, there exists an $S$-homomorphism
$v:T\to A_{i+1}$ such that
\[
    f(t)z=\partial_{i+1}v(t)
    \quad\text{for every }t\in T.
\]
Hence $f(T)$ annihilates the class of $z$ in
$H_i(\mathbf A)$. So $\operatorname{tr}_S(T)H_i(\mathbf A)=0$
for $-1\leq i\leq d$ by the definition of the trace ideal. The homology in all other degrees
is zero.

Assume $|S|,|N|\leq\lambda$. Since $T$ is finitely
generated, $|T|,|\Gamma|,|B|\leq\lambda$.
We may choose the resolution of $B$ above with
$Q_i$ free of rank at most $\lambda$ for $i<d$,
as each successive kernel has cardinality at
most $\lambda$. The $d$th syzygy $Q_d$ is
projective by $\operatorname{pd}_{\Gamma}B\leq d$
and has cardinality at most $\lambda$. For $d=0$,
the same holds with $Q_0=B$. It follows that
each $Q_i$ is a direct summand of
$\Gamma^{(\lambda)}$. Tensoring these decompositions with $T$, we obtain
$A_i=Q_i\otimes_{\Gamma}T$ as a direct summand of
$T^{(\lambda)}$. Since $|T^{(\lambda)}|\leq\lambda$,
each $A_i$ has cardinality at most $\lambda$.
The same bound holds for $H_i(\mathbf A)$, which
is a subquotient of $A_i$ for $0\leq i\leq d$
and a quotient of $N$ for $i=-1$.
\end{proof}

\begin{prop}\label{prop:reduction}
Let $R\twoheadrightarrow S$ be a surjective homomorphism of commutative rings and $\mathbf P$ a complete quasi-projective resolution over $R$. Suppose that $T$ is a finitely generated $S$-module with $\Gamma=\operatorname{End}_S(T)$ and $\mathfrak t=\operatorname{tr}_S(T)$. If every $S/\mathfrak t$-module, regarded as an $R$-module, belongs to $\mathcal E(\mathbf P)$, then every $S$-module $N$ satisfying $\operatorname{pd}_{\Gamma}\operatorname{Hom}_S(T,N)<\infty$ belongs to $\mathcal E(\mathbf P)$.  
\end{prop}

\begin{proof}
Let $\operatorname{pd}_{\Gamma}\operatorname{Hom}_S(T,N)\leq d$ for some integer $d\geq0$.
By Lemma~\ref{lem:trace}, there is a complex
\[
    \mathbf A:\quad
    0\longrightarrow A_d\stackrel{\partial_d}\longrightarrow A_{d-1} \longrightarrow\cdots\longrightarrow A_1\stackrel{\partial_1}
    \longrightarrow A_0\stackrel{\partial_0}\longrightarrow N\longrightarrow0
\]
with $A_i\in\operatorname{Add}_S T$ for $0\leq i\leq d$,
whose homology modules are annihilated by $\mathfrak t$ and hence
are naturally $S/\mathfrak t$-modules. So all $H_i(\mathbf{A})$ belong to $\mathcal{E}(\mathbf{P})$ by hypothesis.
Note that $R\to S$ is surjective. Thus $T$ is finitely generated
over $R$ and $\operatorname{Add}_S T\subseteq\operatorname{Add}_R T$.
Hence $A_i\in\mathcal E(\mathbf P)$ for $0\leq i\leq d$ by Lemma~\ref{lem:coeff}.  The case $d=0$ follows from
$\ker\partial_0=H_0(\mathbf A)\in\mathcal E(\mathbf P)$
and the short exact sequences below. We may assume $d\geq1$. Consider the short exact sequences
\[
    0\longrightarrow \operatorname{Im}\partial_{i}\longrightarrow \operatorname{Ker}\partial_{i-1}\longrightarrow H_{i-1}(\mathbf{A}) \longrightarrow0
    \quad\text{and}\quad
    0\longrightarrow \operatorname{Ker}\partial_{i}\longrightarrow A_i
    \longrightarrow \operatorname{Im}\partial_{i}\longrightarrow0
\] 
for $1\leq i\leq d$.
It is easy to see $\operatorname{Ker}\partial_d=H_d(\mathbf{A})\in \mathcal E(\mathbf P)$. Thus we have $\operatorname{Im}\partial_{d}$ and $\operatorname{Ker}\partial_{d-1}$ both lie in $\mathcal E(\mathbf P)$ by Lemma~\ref{lem:coeff}. Repeating this argument, we obtain
$\operatorname{Ker}\partial_{0} \in\mathcal E(\mathbf P)$. It follows from the following short exact sequences
\[
    0\longrightarrow \operatorname{Ker}\partial_{0}\longrightarrow A_0
    \longrightarrow \operatorname{Im}\partial_{0}\longrightarrow0
  \quad\text{and}\quad
 0\longrightarrow \operatorname{Im}\partial_{0}\longrightarrow N\longrightarrow H_{-1}(\mathbf{A}) \longrightarrow0
 \] 
that $N\in\mathcal E(\mathbf P)$.
\end{proof}

\begin{cor}\label{cor:generator}
Let $R$ be a commutative ring and let $T$ be a finitely
generated generator of $R\text{-}\operatorname{Mod}$.
Set $\Gamma=\operatorname{End}_R(T)$. If
$\operatorname{r.gldim}\Gamma<\infty$, then
$\mathcal{QGP}(R)=\operatorname{Proj}R$.
\end{cor}

\begin{proof}
Since $T$ is a generator, there is an epimorphism
$T^{(\Lambda)}\to R$ for some set $\Lambda$.
Its component maps show that $\operatorname{tr}_R(T)=R$,
so the only module over $R/\operatorname{tr}_R(T)$ is
the zero module.  Let $G\in\mathcal{QGP}(R)$ and choose a complete
quasi-projective resolution
$\mathbf P=(P_i,d_i)_{i\in\mathbb Z}$ with
$G=\operatorname{Coker}d_1$. Since $\operatorname{pd}_{\Gamma}\operatorname{Hom}_R(T,G)<\infty$, it follows from Proposition~\ref{prop:reduction} that $G\in\mathcal E(\mathbf P)$. Hence the short exact sequence
\[
    0\longrightarrow G\xrightarrow{u}P_{-1}
    \longrightarrow\operatorname{Coker}d_0
    \longrightarrow0
\]
is $\operatorname{Hom}_R(-,G)$-exact and splits.
Thus $G$ is projective. The reverse inclusion follows from the
definition.
\end{proof}
We next establish a criterion for projective dimension in terms of Ext vanishing against a finite family of quasi-projective modules.

Set $S_0=R$, and suppose that there is a finite sequence $S_0,S_1,\ldots,S_\ell$ of commutative rings such that, for each $0\le j<\ell$, either $S_{j+1}=S_j/\operatorname{tr}_{S_j}(T_j)$, where $T_j$ is a nonzero finitely generated $S_j$-module satisfying
\[\operatorname{r.gldim}\operatorname{End}_{S_j}(T_j)\le d_j<\infty,
\]
or $S_{j+1}=S_j/J_j$ for a nilpotent ideal $J_j$ of $S_j$. Assume that $S_\ell$ is Artinian, possibly the zero ring, and let $\mathcal T$ denote the set of indices $j$ for which the first alternative occurs.  Let 
\[c=\sum_{j\in\mathcal T}(d_j+1),\] with $c=0$ if $\mathcal T=\varnothing$. All $S_j$-modules are regarded as $R$-modules by restriction of scalars. For a set $X$, let $|X|$ denote its cardinality.

For each $j\in\mathcal T$, $T_j$ is finitely generated as an $R$-module. As in the proof of Lemma \ref{lem:coeff}, there is a finite filtration
\[
0=T_{j,0}\subseteq T_{j,1}\subseteq\cdots\subseteq T_{j,s_j}=T_j
\]
such that $T_{j,u}/T_{j,u-1}\cong R/I_{j,u}$ for some ideal $I_{j,u}$ of $R$. Let $R/\mathfrak m_1,\ldots,R/\mathfrak m_v$ represent the isomorphism classes of simple $S_\ell$-modules. For each infinite cardinal $\lambda\ge |R|$, set
\[
\mathscr Q_\lambda=
\bigl\{(R/I_{j,u})^{(\lambda)}
\mid j\in\mathcal T,\ 1\le u\le s_j\bigr\}
\cup
\bigl\{(R/\mathfrak m_a)^{(\lambda)}
\mid 1\le a\le v\bigr\}.
\]
Thus $\mathscr Q_\lambda$ is a finite family of quasi-projective $R$-modules.

\begin{lem}\label{lem:ext-range}
Let $t$ be a nonnegative integer, $X$ an
$R$-module and
\[
    \mathbf A:\quad
    0\longrightarrow A_d\stackrel{\partial_d}\longrightarrow A_{d-1} \longrightarrow\cdots\longrightarrow A_1\stackrel{\partial_1}
    \longrightarrow A_0\stackrel{\partial_0}\longrightarrow N\longrightarrow0
\]
be a complex. If $\operatorname{Ext}_R^{t+i}(X,A_i)=0$ for $0\leq i\leq d$
and $\operatorname{Ext}_R^{t+i+1}(X,H_i(\mathbf A))=0$
for $-1\leq i\leq d$, then
$\operatorname{Ext}_R^t(X,N)=0$.
\end{lem}

\begin{proof}
For $0\leq i\leq d$, there are short exact sequences
\[
    0\longrightarrow\operatorname{Ker}\partial_i
    \longrightarrow A_i
    \longrightarrow\operatorname{Im}\partial_i
    \longrightarrow0\quad\text{and}\quad  0\longrightarrow\operatorname{Im}\partial_{i+1}
    \longrightarrow\operatorname{Ker}\partial_i
    \longrightarrow H_i(\mathbf A)
    \longrightarrow0,
\]
where $\operatorname{Im}\partial_{d+1}=0$. Note that $\operatorname{Ker}\partial_d=H_d(\mathbf A)$
and these sequences imply by descending induction that
$\operatorname{Ext}_R^{t+i}(X,\operatorname{Im}\partial_i)=0$
for $0\leq i\leq d$.
So we have the following exact sequence
\[
    0=\operatorname{Ext}_R^{t}(X,\operatorname{Im}\partial_0)
    \longrightarrow \operatorname{Ext}_R^{t}(X,N)
    \longrightarrow \operatorname{Ext}_R^{t}(X,H_{-1}(\mathbf A))=0.
\]
Therefore, $\operatorname{Ext}_R^{t}(X,N)=0$.
\end{proof}

\begin{prop}\label{lem:finite-test-vanishing}
Let $r$ be a positive integer, $\lambda\geq |R|$ an
infinite cardinal and $X$ an $R$-module. If
$\operatorname{Ext}_R^q(X,Q)=0$ for any
$Q\in\mathscr Q_\lambda$ and $r\leq q\leq r+c$,
then $\operatorname{Ext}_R^r(X,N)=0$ for any
$R$-module $N$ with $|N|\leq\lambda$.
\end{prop}

\begin{proof}
For $0\leq j\leq\ell$, let
\[
    \rho_j=\sum_{\substack{k\in\mathcal T\\k<j}}(d_k+1).
\]
Since $\mathcal T\subseteq\{0,\ldots,\ell-1\}$,
$\rho_0=0$ and $\rho_\ell=c$. We prove by descending
induction on $j$ that
\[
    \operatorname{Ext}_R^q(X,N)=0,\quad r\leq q\leq r+\rho_j
\]
for every $S_j$-module $N$ with $|N|\leq\lambda$.

For $j=\ell$, the assertion is clear if $S_\ell=0$.
Otherwise, let $J$ be the Jacobson radical of $S_\ell$.
Since $J$ is nilpotent, there is an integer $m\geq1$
such that
\[
    N\supseteq JN\supseteq\cdots\supseteq J^mN=0.
\]
For $0\leq k<m$, the quotient $J^kN/J^{k+1}N$ is
annihilated by $J$ and has cardinality at most
$\lambda$. Since $S_\ell/J$ is semisimple, $J^kN/J^{k+1}N$
is a direct summand of
$\bigoplus_{a=1}^{v}(R/\mathfrak m_a)^{(\lambda)}$. 
Consequently, $\operatorname{Ext}_R^q(X,J^kN/J^{k+1}N)=0$
for $0\leq k<m$ and $r\leq q\leq r+c$.
The short exact sequences
\[
0\longrightarrow J^{k+1}N\longrightarrow J^kN
\longrightarrow J^kN/J^{k+1}N\longrightarrow0,
\quad0\leq k<m
\]
and $J^mN=0$ imply
$\operatorname{Ext}_R^q(X,J^kN)=0$ for $0\leq k\leq m$
and $r\leq q\leq r+c$ by descending induction on $k$.
In particular, $\operatorname{Ext}_R^q(X,N)=0$ for
$r\leq q\leq r+\rho_\ell$, since $\rho_\ell=c$.

Suppose the assertion holds at $j+1$, and let $N$ be an
$S_j$-module with $|N|\leq\lambda$. If $j\notin\mathcal T$, then $S_{j+1}=S_j/J_j$
and $J_j^m=0$ for some $m\geq1$. For $0\leq k<m$,
the quotient $J_j^kN/J_j^{k+1}N$ is an
$S_{j+1}$-module of cardinality at most $\lambda$.
Thus 
\[\operatorname{Ext}_R^q
(X,J_j^kN/J_j^{k+1}N)=0\; \text{for}\;
r\leq q\leq r+\rho_{j+1}
\] by the induction hypothesis.
The short exact sequences
\[
0\longrightarrow J_j^{k+1}N\longrightarrow J_j^kN
\longrightarrow J_j^kN/J_j^{k+1}N\longrightarrow0,
\quad 0\leq k<m
\]
and $J_j^mN=0$ imply by descending induction that
$\operatorname{Ext}_R^q(X,N)=0$ for
$r\leq q\leq r+\rho_{j+1}$.
Since $\rho_j=\rho_{j+1}$, the assertion follows for $j$.

If $j\in\mathcal T$, there exists a complex 
\[
    \mathbf A:\quad
    0\longrightarrow A_{d_j}\longrightarrow\cdots
    \longrightarrow A_0\longrightarrow N
    \longrightarrow0
\]
with $A_i$ a direct summand of $T_j^{(\lambda)}$
for $0\leq i\leq d_j$ and $H_i(\mathbf A)$ an
$S_{j+1}$-module of cardinality at most $\lambda$
for $-1\leq i\leq d_j$ by Lemma~\ref{lem:trace}. The filtration of $T_j$ induces short exact sequences
\[
0\longrightarrow T_{j,u-1}^{(\lambda)}
\longrightarrow T_{j,u}^{(\lambda)}
\longrightarrow (R/I_{j,u})^{(\lambda)}
\longrightarrow0,
\quad 1\leq u\leq s_j.
\]
Since $T_{j,0}^{(\lambda)}=0$ and
$(R/I_{j,u})^{(\lambda)}\in\mathscr Q_\lambda$,
the long exact Ext sequences imply by induction on $u$
that 
\[\operatorname{Ext}_R^q(X,T_j^{(\lambda)})=0\;\text{for}\; r\leq q\leq r+c.\] As $A_i$ is a direct summand
of $T_j^{(\lambda)}$, it follows that
\[
\operatorname{Ext}_R^q(X,A_i)=0 \; \text{for}\; 0\leq i\leq d_j\;\text{and}\; r\leq q\leq r+c.\]
By the induction hypothesis, we have
\[\operatorname{Ext}_R^q(X,H_i(\mathbf A))=0 \;\text{for}\; -1\leq i\leq d_j\; \text{and}\;
r\leq q\leq r+\rho_{j+1}.\] For $r\leq t\leq r+\rho_j$, the equality
$\rho_{j+1}=\rho_j+d_j+1\leq c$ implies
$r\leq t+i\leq r+c$ for $0\leq i\leq d_j$ and
$r\leq t+i+1\leq r+\rho_{j+1}$ for
$-1\leq i\leq d_j$. It follows from Lemma~\ref{lem:ext-range} that
$\operatorname{Ext}_R^t(X,N)=0$. Thus the assertion holds for $j\in\mathcal T$.

The cases $j\notin\mathcal T$ and $j\in\mathcal T$
complete the descending induction. Since $S_0=R$
and $\rho_0=0$, we have
$\operatorname{Ext}_R^r(X,N)=0$ for every $R$-module
$N$ with $|N|\leq\lambda$.
\end{proof}

\begin{thm}\label{thm:finite-test}
Let $\lambda\geq|R|$ be an infinite cardinal and $X$ an
$R$-module with $|X|\leq\lambda$. For any integer
$n\geq0$, $\operatorname{pd}_R X\leq n$ if and only if
$\operatorname{Ext}_R^{n+i}(X,Q)=0$ for all
$Q\in\mathscr Q_\lambda$ and $1\leq i\leq c+1$.
In particular, $\mathcal{QGP}(R)=\operatorname{Proj}R$.
\end{thm}
\begin{proof}
If $\operatorname{pd}_R X\leq n$, then
$\operatorname{Ext}_R^{n+i}(X,Q)=0$ for any
$R$-module $Q$ and any $i\geq1$. Conversely, suppose that
$\operatorname{Ext}_R^{n+i}(X,Q)=0$ for all
$Q\in\mathscr Q_\lambda$ and $1\leq i\leq c+1$.
Since $|R|,|X|\leq\lambda$, there is a free resolution
\[
\cdots\longrightarrow F_1\xrightarrow{\partial_1}F_0
\xrightarrow{\partial_0}X\longrightarrow0
\]
with $F_i$ of rank at most $\lambda$ for all $i$.
In particular,
$|\operatorname{Ker}\partial_n|\leq\lambda$.
By Proposition~\ref{lem:finite-test-vanishing}
for $r=n+1$, we have
$\operatorname{Ext}_R^{n+1}
(X,\operatorname{Ker}\partial_n)=0$.
It follows from dimension shifting that
\[
\operatorname{Ext}_R^1
(\operatorname{Im}\partial_n,\operatorname{Ker}\partial_n)
\cong
\operatorname{Ext}_R^{n+1}
(X,\operatorname{Ker}\partial_n)=0.
\]
Hence the short exact sequence
\[
0\longrightarrow\operatorname{Ker}\partial_n
\longrightarrow F_n\longrightarrow
\operatorname{Im}\partial_n\longrightarrow0
\]
splits. Thus $\operatorname{Im}\partial_n$ is projective,
and $\operatorname{pd}_R X\leq n$.

Let $G\in\mathcal{QGP}(R)$, and let $\mu$ be an
infinite cardinal with $|R|,|G|\leq\mu$.
Note that every $Q\in\mathscr Q_\mu$ is
quasi-projective. By Proposition~\ref{TTTT}\textup{(1)},
we have $\operatorname{Ext}_R^i(G,Q)=0$ for all
$Q\in\mathscr Q_\mu$ and $i\geq1$.
The first assertion for $n=0$ and $\lambda=\mu$
implies that $G$ is projective. The reverse inclusion
follows from the definition.
\end{proof}
\begin{exam}\label{ex:node-test}
Let $k$ be a field, $R=k[[x,y]]/(xy)$ and
$\mathfrak m=(x,y)$. For every infinite cardinal
$\lambda\geq|R|$, set
\[
\mathscr Q_\lambda=
\bigl\{(R/(x))^{(\lambda)},
(R/(y))^{(\lambda)},
(R/\mathfrak m)^{(\lambda)}\bigr\}.
\]
Let $X$ be an $R$-module with $|X|\leq\lambda$
and $n$ a nonnegative integer. Then
$\operatorname{pd}_R X\leq n$ if and only if
$\operatorname{Ext}_R^{n+i}(X,Q)=0$ for all
$Q\in\mathscr Q_\lambda$ and $1\leq i\leq3$.
\end{exam}
\begin{proof}
Let $T=\mathfrak m$. We have
$\mathfrak m=(x)\oplus(y)$, where
\[
(x)\cong R/(y)\cong k[[x]]
\quad\text{and}\quad
(y)\cong R/(x)\cong k[[y]].
\]
The relation $xy=0$ implies that the image of any
homomorphism $(x)\to(y)$ is annihilated by $y$.
Since multiplication by $y$ is injective on $(y)$,
$\operatorname{Hom}_R((x),(y))=0$.
Similarly, $\operatorname{Hom}_R((y),(x))=0$.
It follows that
\[
\operatorname{End}_R(T)
\cong
\operatorname{End}_R((x))\times
\operatorname{End}_R((y))
\cong k[[x]]\times k[[y]].
\]
Since $k[[x]]$ and $k[[y]]$ are discrete valuation
rings, $\operatorname{r.gldim}\operatorname{End}_R(T)=1$.
For $f\in\operatorname{Hom}_R(T,R)$, $R$-linearity
implies $yf(x)=f(yx)=0$ and $xf(y)=f(xy)=0$.
Hence $f(x)\in\operatorname{ann}_R(y)=(x)$ and
$f(y)\in\operatorname{ann}_R(x)=(y)$.
Since $T=(x)+(y)$, $f(T)\subseteq\mathfrak m$.
The inclusion $T\longrightarrow R$ has image
$\mathfrak m$, and hence
$\operatorname{tr}_R(T)=\mathfrak m$.

The filtration $0\subset(x)\subset T$ has cyclic
factors $R/(y)$ and $R/(x)$, and
$R/\operatorname{tr}_R(T)\cong k$.
Thus the trace step $R\twoheadrightarrow k$ has
$d_0=1$ and $c=2$, with associated test family
$\mathscr Q_\lambda$.
If $\operatorname{pd}_R X\leq n$, the stated Ext groups
vanish. The reverse implication follows from
Theorem~\ref{thm:finite-test}.
\end{proof}

The preceding results establish
$\mathcal{QGP}(R)=\operatorname{Proj}R$
for several classes of rings. We conclude with
the following question.

\begin{problem}\label{prob:nonprojective-qgp}
Does there exist an associative ring $R$ such that
$\mathcal{QGP}(R)\ne\operatorname{Proj}R$?
\end{problem}

If a nonprojective $G\in\mathcal{QGP}(R)$ exists,
then $\QpdR{G}=\infty$ by
Proposition~\ref{comparison}. For commutative $R$,
Corollary~\ref{cor:fg-qgp} further implies that
$G$ is not finitely generated.

	\end{document}